\documentclass{siamart190516}
\usepackage{amsmath,amssymb,graphicx}
\usepackage[a4paper, total={6in, 8in}]{geometry}
\usepackage{multirow}
\usepackage{blkarray}
\usepackage{dsfont}
\usepackage[draft]{todonotes}
\usepackage{algorithm,algpseudocode}
\usepackage{tikz-cd}

\DeclareMathOperator*{\argmin}{arg\min}
\graphicspath{{Figures/}}
\title{Relaxed regularization for linear inverse problems
\thanks{Submitted to the editors 24-6-2020.
\funding{NL is financially supported by the DELPHI consortium. TvL is financially supported by the Netherlands Organization for Scientific Research (NWO) as part of research programme 613.009.032. }}
}
\author{
Nick Luiken \thanks{Utrecht University, Mathematical Institute} (\email{n.a.luiken@uu.nl})
\and
Tristan van Leeuwen \thanks{Utrecht University, Mathematical Institute} (\email{t.vanleeuwen@uu.nl})
}
\date{\today}
\graphicspath{{figures/}}
\begin{document}
\maketitle

\begin{abstract}
	We consider regularized least-squares problems of the form $\min_{x} \frac{1}{2}\Vert Ax - b\Vert_2^2 + \mathcal{R}(Lx)$. Recently, Zheng et al. \cite{Zheng2019} proposed an algorithm called Sparse Relaxed Regularized Regression (SR3) that employs a splitting strategy by introducing an auxiliary variable $y$ and solves $\min_{x,y} \frac{1}{2}\Vert Ax - b\Vert_2^2 + \frac{\kappa}{2}\Vert Lx - y\Vert_2^2 + \mathcal{R}(x)$. By minimizing out the variable $x$, we obtain an equivalent optimization problem $\min_{y} \frac{1}{2} \Vert F_{\kappa}y - g_{\kappa}\Vert_2^2+\mathcal{R}(y)$. In our work, we view the SR3 method as a way to approximately solve the regularized problem. We analyze the conditioning of the relaxed problem in general and give an expression for the SVD of $F_{\kappa}$ as a function of $\kappa$.
	Furthermore, we relate the Pareto curve of the original problem to the relaxed problem and we quantify the error incurred by relaxation in terms of $\kappa$. Finally, we propose an efficient iterative method for solving the relaxed problem with inexact inner iterations. Numerical examples illustrate the approach.\\
\end{abstract}



\section{Introduction}
Inverse problems are problems where a certain quantity of interest has to be determined from indirect measurements. In medicine, well-known examples include MRI \cite{Liang2000}, CT \cite{Herman2009}, and ultrasound imaging \cite{Besson2019} where the objective is to obtain images of the interior of the human body. In the geosciences, inverse problems arise in seismic exploration and seismology \cite{Virieux2009}, where the interest lies in exploring the elastic properties of the different layers of our planet. Other examples include tomography \cite{Arridge2009,Natterer2001,Bertero1998}, radar imaging \cite{Borden2001}, remote sensing \cite{Tamminen2012,Nolan2002}, astrophysics \cite{Starck2016}, and more recently, machine learning \cite{Goodfellow2016}.

Inverse problems are challenging for a number of reasons. There may be limited data available, or the data may be corrupted by noise. The datasets are generally very large, and the underlying model is generally not well-defined for retrieving the quantity of interest. Therefore, inverse problems often have to be \textit{regularized}, meaning prior information has to be added. They can be posed in the following way:
\begin{equation}
\min_{x} \dfrac{1}{2}\Vert Ax - b\Vert_2^2 +  \mathcal{R}(Lx),
\label{eq:rlsqrs}
\end{equation}
where $A \in \mathbb{R}^{m\times n}$ is the linear forward operator, $\mathcal{R}(\cdot)$ is the regularization term and $L \in \mathbb{R}^{n\times p}$ the regularization operator. The latter two encode the prior information about $x$. In our work, we focus on $\mathcal{R}(\cdot) = \lambda\|\cdot\|_p^p$, or, equivalently, $\mathcal{R}(\cdot) = \delta_{\|\cdot\|_p\leq\tau}(\cdot)$, which is the indicator function of the set $\Vert \cdot\Vert_p \leq \tau$. By equivalent we mean that for every $\tau$ there is a $\lambda$ such that the solutions of the two problems coincide \cite{Aravkin2012e}. A direct solution to the problem above is generally not possible, either because a closed-form solution does not exist, or because evaluating the direct solution is too computationally expensive. Therefore, we have to resort to iterative methods to solve the problem, with most algorithms being designed for specific choices of $p$ and $L$. \\

Traditionally, $p = 2$, called Tikhonov regularization, is a popular choice, because the objective function is differentiable and allows for a closed-form expression of the solution of \cref{eq:rlsqrs} in terms of $A, L$ and $\lambda$. For this class of problems, Krylov based algorithms have been proven very effective \cite{Calvetti2000,Calvetti2003,Hansen2010,Kilmer2001,Golub1997,Zwaan2017,Hochstenbach2010,Hochstenbach2015,Kilmer2007,Gazzola2015}. These methods generally exploit the fact that a closed-form solution exists by constructing a low dimensional subspace from which an approximate solution is extracted.

The choice $p = 1$ has gained popularity in recent years because it gives sparse solutions while still yielding a convex objective. Sparsity is important in a number of applications, like compressed sensing \cite{Candes2006}, seismic imaging \cite{Herrmann2008}, image restoration \cite{Rudin1994}, and tomography \cite{Hansen2011}. However, the objective is no longer differentiable and the aforementioned Krylov methods do not apply. If $L = I$, a proximal gradient method (sometimes referred to as Iterative Soft Tresholding -- ISTA) \cite{Daubechies2004} can be applied, iteratively updating the solution via
\[
	x_{k+1} = \text{prox}_{\alpha\lambda\|\cdot\|_1}\left(x_k - \alpha A^T\!(Ax_{k} - b)\right),
\]
where $\alpha\in (0,\|A\|^2_2)$ is the stepsize and the proximal operator is the soft thresholding operator, which can be efficiently evaluated. Generally, ISTA achieves a sub-linear rate of convergence of $\mathcal{O}(1/k)$ (unless $m \geq n$ and $A$ has full rank, in which case we have a linear rate of convergence).
FISTA (Fast Iterative Soft Thresholding Algorithm) \cite{Beck2009} is a faster version of ISTA that generally achieves a sublinear rate of $\mathcal{O}(1/k^2)$.\\

If $L = I$ the optimization problem is said to be in standard-form and for any other $L$ the algorithm is in general form. If $L$ is full-rank and has no nullspace, the optimization problem can be put into standard-form via the change of variables $y = Lx$. Instead of the matrix $A$, we get $AL^{\dagger}$. In such cases we can apply the (F)ISTA method directly at the expense of having to evaluate $L^\dagger$. In some applications, we have $L^\dagger = L^T$ (e.g., when $L$ is a tight frame). If $L$ has a non-trivial nullspace the algorithm can still be put in standard-form by the standard-from transformation \cite{Elden1982,Hansen2011}, but this is nontrivial, because the nullspace has to be accounted for. \\

If $L \neq I$, and we cannot easily transform the problem to standard form, the proximal operator is no longer easy to evaluate in general and FISTA may no longer be attractive. An example of this class of problems is Total Variation (TV) regularization, where $L$ is the discretization of the gradient, which gives blocky solutions. A popular algorithm for this class of problems is the Alternating Direction Method of Multipliers, ADMM \cite{Boyd2011}. ADMM solves \cref{eq:rlsqrs} by forming the \textit{augmented Lagrangian}
\[
	\min_{x,y}\max_z \dfrac{1}{2}\Vert Ax - b\Vert_2^2 + \lambda\Vert y\Vert_p^p + z^T\left(Lx - y\right) + \dfrac{\rho}{2} \Vert Lx - y\Vert_2^2,
\]
and alternatingly minimizing over the variables $x$ and $y$, and the Lagrange multiplier $z$. The strength of ADMM is that it can closely approximate the solution of any convex sparse optimization problem. However, convergence can be slow \cite{Boyd2011}. \\

If $p < 1$, the emphasis on sparsity of the solution is stronger than for the case $p = 1$. However, the objective function is no longer convex which makes it more difficult to solve. \\

Recently, a unifying algorithm was proposed that allows the efficient approximation of the solution of any problem of the form \cref{eq:rlsqrs}, called Sparse Relaxed Regularized Regression (SR3) \cite{Zheng2019}. This algorithm makes use of a splitting strategy by introducing an auxiliary variable $y$ and yields:
\begin{equation}
	\min_{x,y} \dfrac{1}{2}\|Ax - b\|_2^2 + \dfrac{\kappa}{2}\|Lx - y\|_2^2 + \mathcal{R}(y).
	\label{eq:relaxed}
\end{equation}
By minimizing out $x$, we obtain a new optimization problem of the form:
\begin{equation}
	\bar{y}_{\kappa} = \argmin_{y}\dfrac{1}{2}\| F_{\kappa}y - g_{\kappa}\|_2^2 + \mathcal{R}(y),
\end{equation}
where $F_{\kappa} = \left(\begin{matrix}\kappa^{1/2}\left(I - \kappa LH_{\kappa}^{-1}L^T\right)  \\ \kappa AH_{\kappa}^{-1}L^T \end{matrix}\right)$ and $g_{\kappa} = \left(\begin{matrix}\kappa^{1/2}LH_{\kappa}^{-1}A^Tb \\ b - AH_{\kappa}^{-1}A^Tb \end{matrix}\right)$, $H_{\kappa} = A^T\!A + \kappa L^T\!L$. The solution to \eqref{eq:relaxed} is then given by
\begin{equation}
	\bar{x}_{\kappa} = H_{\kappa}^{-1}\left( \kappa L^T\bar{y}_{\kappa} + A^Tb \right).
\label{eq:SR3x}
\end{equation}
This solution is then used as an approximation of the solution of \eqref{eq:rlsqrs}. In \cite{Zheng2019} the particular case with $L^TL = I$ is analyzed. Using the SVD of $A$, the singular values of $F_{\kappa}$ were calculated, showing a relation between the condition number of $F_{\kappa}$ and $A$ depending on $\kappa$. In short, the result shows that a small $\kappa$ improves the conditioning of $F_{\kappa}$ and as $\kappa\to \infty$ the condition numbers are the same, because the original optimization problem is obtained.

For the implementation of SR3, it is not necessary to form the operator $F_{\kappa}$, as was shown in \cite{Zheng2019}. The authors propose the following algorithm for solving the relaxed problem
\begin{eqnarray}
	x_{k+1} &\leftarrow& \left(A^T\!A + \kappa L^T\!L\right)^{-1} \left(A^Tb + \kappa L^Ty_k\right) \\
	y_{k+1} &\leftarrow& \text{prox}_{\alpha\mathcal{R}} \left(y_k - \alpha\kappa(y_k - Lx_{k+1}) \right),
\end{eqnarray}
which for the particular choice $\alpha = 1/\kappa$ simplifies to
\begin{align}
	x_{k+1} &\leftarrow \left(A^T\!A + \kappa L^T\!L\right)^{-1} \left(A^Tb + \kappa L^Ty_k\right) \label{eq:implementation_step1}  \\
	y_{k+1} &\leftarrow \text{prox}_{1/\kappa \mathcal{R}} \left(Lx_{k+1}\right). \label{eq:implementation_step2}
\end{align}
This method has several advantages when applied to solving inverse problems that we highlight in the examples below.

\subsection{Motivating examples}
\label{motivating_examples}
Below we show some typical examples encountered in various areas of science to which SR3 can be applied. The problems we tackle are of the form
\begin{equation}
\min_x \dfrac{1}{2}\|Ax - b\|_2^2 \quad \text{s.t.} \quad \|Lx\|_1 \leq \tau.
\label{eq:lstau}
\end{equation}
The main tasks are to solve this for a given value of $\tau$ and to find an appropriate value of $\tau$. The latter is achieved by picking the corner of the Pareto curve (sometimes called the L-curve) $\phi(\tau) = \min_{\|x\|_p\leq \tau}\|Ax - b\|_2$.
Comparing a proximal gradient method to SR3, we show the residual as a function of $\tau$, the optimal reconstruction, and the convergence history in terms of the primal-dual gap. These examples show two favourable aspects of SR3 over the conventional proximal gradient method: \emph{i)} SR3 converges (much) faster for any fixed value of $\tau$ and \emph{ii)} the corners of both Pareto-curves coincide, allowing us to effectively use SR3 to estimate $\tau$.

\subsubsection*{Spiky deconvolution ($m = n$, $L = I$)}
Consider a deconvolution problem where $A$ is a Toeplitz-matrix that convolves the input with a bandlimited function;
\[
a_{ij} = w(t_i - t_j),
\]
where $w(t) = (1 - (t/\sigma)^2)e^{-(t/\sigma)^2} $ and $t_i = i\cdot h$. We take $n = 101$, $h = 1/n$ and $\sigma = 0.05$. The results are shown in figure \ref{fig:example1}.
\begin{figure}
\centering
\includegraphics[scale=.4]{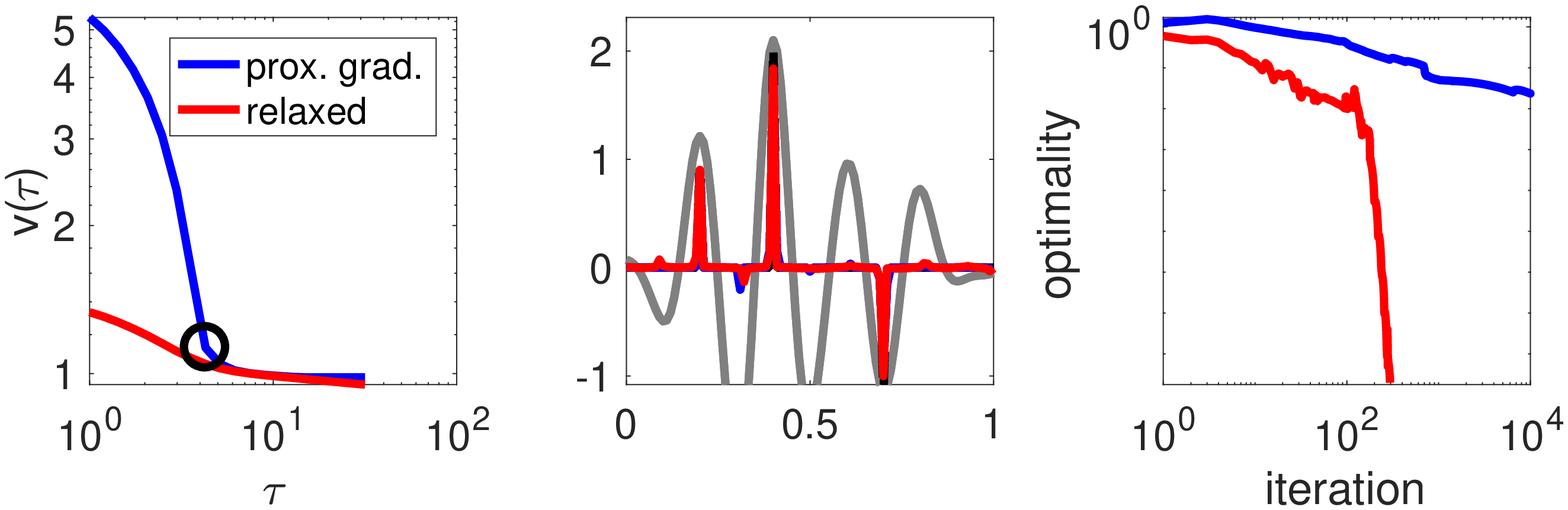}
\label{fig:example1}
\caption{Spiky deconvolution example. The left figure shows the Pareto curve, the middle figure shows the solution and the right figure shows the primal-dual gap as a function of the number of iterations. The grey line in the middle figure shows the minimum norm solution.}
\end{figure}

\subsubsection*{Compressed sensing ($m < n$, $L = I$)}
Here, the goal is to recover a sparse signal from compressive samples. The forward operator is a random matrix with i.i.d. normally distributed entries. We take $n = 101$ and $m = 20$. The results are shown in figure \ref{fig:example2}.

\begin{figure}
\centering
\includegraphics[scale=.4]{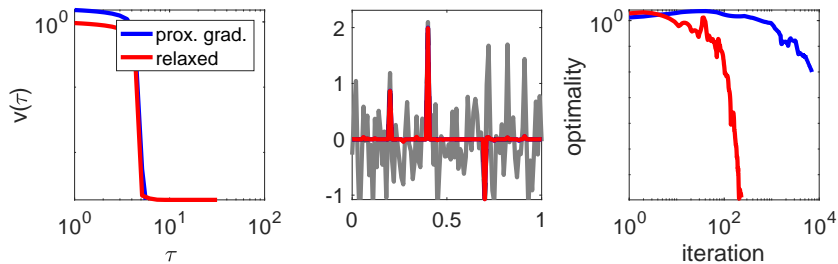}
\label{fig:example2}
\caption{Compressed sensing. A signal is reconstructed from very few samples, which requires sparse regularization. The left figure shows the Pareto curve, the middle figure shows the solution and the right figure shows the primal-dual gap as a function of the number of iterations. The grey line in the middle figure shows the minimum norm solution.}
\end{figure}

\subsubsection*{Total variation ($m = n$, $L = D$)}
Consider a deconvolution problem where $A$ is a Toeplitz-matrix that convolves the input with a bandlimited function;
\[
a_{ij} = w(t_i - t_j),
\]
where $w(t) = e^{-(t/\sigma)^2}$ and $t_i = i\cdot h$. $L$ is a finite-difference discretization of the first-order derative with Neumann boundary conditions. We take $n = 101$, $h = 1/n$ and $\sigma = 0.05$. The results are shown in figure \ref{fig:example3}.
\begin{figure}
\centering
\includegraphics[scale=.4]{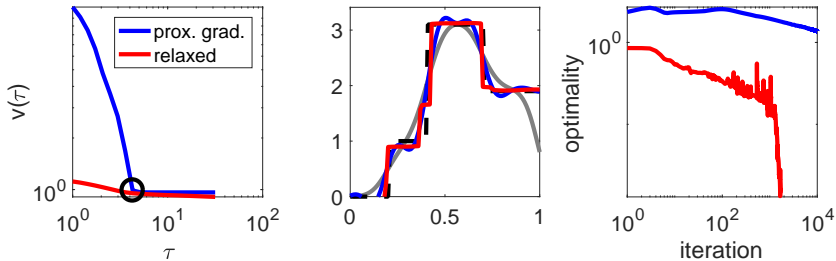}
\label{fig:example3}
\caption{Total variation example. Here, the solution has a blocky structure. The left figure shows the Pareto curve, the middle figure shows the solution and the right figure shows the primal-dual gap as a function of the number of iterations. The grey line in the middle figure shows the minimum norm solution.}
\end{figure}

\subsection{Contributions} In this paper we set out to further analyze the SR3 method proposed in \cite{Zheng2019} and analyze in detail the observations made in the above examples. Our contributions are:

\begin{description}
\item[Conditioning of $F_{\kappa}$ for general $L$.] We extend the analysis of \cite{Zheng2019} and derive the SVD of $F_{\kappa}$ for general $L$. We show how the singular values and the condition number of $F_{\kappa}$ are related to the generalized singular values of $(A,L)$. As a by-product, we show that SR3 implicitly makes a standard-form transformation {\cite{Elden1982}} of \cref{eq:rlsqrs}.

\item[Approximation of the Pareto-curve.] We show that that the Pareto curve corresponding to the relaxed problem \eqref{eq:relaxed} always underestimates the Pareto curve of the original problem \eqref{eq:rlsqrs} and that the error is of order $\mathcal{O}(\kappa^{-2})$. A by-product of this result is a better understanding of the Pareto curve for general $p$ and an intuitive explanation of the observation that the corners of the relaxed original Pareto curves coincide.

\item[Inexact solves.] We propose an inexact inner-outer iterative version of the SR3 algorithm where the regularized least-squares problem \cref{eq:implementation_step1} is solved approximately using a Krylov-subspace method. In particular, we propose an automated adaptive stopping criterion for the inner iterations.

\end{description}

\subsection{Outline}
In \cref{sec:analysis} we analyze the operator $F_{\kappa}$. We derive the SVD of $F_{\kappa}$ and analyse the limiting cases $\kappa\to \infty$ and $\kappa\to 0$. Our main results are a characterization of the singular values of $F_{\kappa}$ and showing that SR3 implicitly applies a standard-form transformation. In \cref{sec:value_function}, we relate the Pareto curve of SR3 to the Pareto curve of the original problem and derive an error bound in terms of $\kappa$. Next, \cref{sec:implementation} is concerned with the implementation of SR3. We propose two ingredients that make SR3 suitable for large-scale applications. In \cref{sec:experiments}, we conduct our numerical experiments and verify the theoretical results from \cref{sec:analysis}. Moreover, we numerically investigate the influence of $\kappa$ on the convergence rate. Finally, in \cref{sec:conclusion}, we draw our conclusions.

\section{Analysis of SR3}
\label{sec:analysis}
In this section we analyze some of the properties of the operator $F_{\kappa}$. We will characterize the singular values of $F_{\kappa}$ for general $L$ and analyse the limits $\kappa\to 0$ and $\kappa\to \infty$. First, we will treat some preliminaries needed for understanding what happens in the limit $\kappa\to\infty$.

\subsection{The Generalized Singular Value Decomposition}
The central tool in our analysis is the Generalized Singular Value Decomposition (GSVD) of $(A,L)$. The definition of the GSVD depends on the size of the matrices and the dimensions of the matrices relative to each other. We use the definitions for the case $A\in\mathbb{R}^{m\times n}$ and $L\in\mathbb{R}^{p\times n}$ where $m \geq n$, $p < n$ or $m < n$, $p > n$ because this corresponds to the examples we use in our experiments. 
\begin{definition}[GSVD]
	Let $A\in\mathbb{R}^{m\times n}$ and $L\in\mathbb{R}^{p\times n}$. The Generalized Singular Value Decomposition (GSVD) of $(A,L)$ is given by $A = U\Sigma X$, $L = V\Gamma X$, where
	$$\Sigma = \begin{bmatrix} \Sigma_p & 0 \\ 0 & I_{n-p} \\ 0 & 0\end{bmatrix}, \, \quad \Gamma = \begin{bmatrix} \Gamma_p & 0 \end{bmatrix} \quad \text{for} \quad m \geq n, p \leq n,$$
	and
	$$\Sigma = \begin{bmatrix} 0 & \Sigma_m \end{bmatrix}, \, \quad \Gamma = \begin{bmatrix} I_{n-m} & 0 \\ 0 & \Gamma_m \\ 0 & 0\end{bmatrix} \quad \text{for}\quad  m < n, p > n.$$
	The matrices $\Sigma_r$ and $\Gamma_r$ (where $r = p$ or $r = m$) are $r\times r$ diagonal matrices satisfying $\Sigma_r^T\Sigma_r + \Gamma_r^T\Gamma_r = I_r$, $X$ is invertible and $U$ and $V$ are orthonormal. Moreover, we have the following ordering of the diagonal elements $\sigma_i$ of $\Sigma$ and $\gamma_i$ of $\Gamma$:
\begin{eqnarray*}
	& 0 \leq \gamma_{r} \leq \ldots \leq \gamma_1 \leq 1, \\
	& 0 \leq \sigma_{1} \leq \ldots \leq \sigma_r \leq 1.
\end{eqnarray*}
\end{definition}

The decomposition of $A$ and $L$ in the GSVD share similar properties to the SVD. The number of nonzero entries of $\Sigma$ and $\Gamma$ give the rank of $A$ and $L$ respectively. If $r_{A}$ is the rank of $A$ and $r_{L}$ is the rank of $L$ then the last $r - r_{A}$ columns, corresponding to $\Sigma_r$, of $U$ form a basis for the range of $A$ and the first $r_{L}$ columns, corresponding to $\Gamma_r$, of $V$ form a basis for the range of $L$. The first $r - r_{A}$ columns, corresponding to $\Sigma_r$, of $X^{-1}$ form a basis for the nullspace of $A$ and the last $r - r_{L}$ columns, corresponding to $\Gamma_r$, of $X^{-1}$ form a basis for the nullspace of $L$.

\subsection{Standard-form transformation}
The standard-form transformation, see e.g. \cite{Elden1982,Hansen1998}, makes a substitution $y = Lx$ such that $x = x_{\mathcal{M}} + x_{\mathcal{N}}$, where
\begin{equation}
	\bar{x}_{\mathcal{M}} = L_A^{\dagger}\bar{y}, \quad \bar{y} = \argmin_{y} \dfrac{1}{2}\Vert AL_{A}^{\dagger}y - b \Vert_2^2 + \mathcal{R}(y), \quad L_{A}^{\dagger} = \left(I - (A(I-L^{\dagger}L))^{\dagger}A\right)L^{\dagger}.
\label{eq:xM1}
\end{equation}
and
\begin{equation}
	\bar{x}_{\mathcal{N}} = \left( A\left(I - L^{\dagger}L\right)\right)^{\dagger}b.
\label{eq:xN1}
\end{equation}
The operator $L_{A}^{\dagger}$ is called the \textit{A-weighted pseudo-inverse}. The transformation splits the solution into two parts: one part in the range of $L$, $L_A^{\dagger}y$, and one part in the nullspace of $L$, $x_{\mathcal{N}}$. The operator $L_{A}^{\dagger}$ makes the two parts $A$-orthogonal. The parts $L_A^{\dagger}y$ and $x_{\mathcal{N}}$ are then obtained by two independent optimization problems. If $L$ is invertible $L_{A}^{\dagger} = L^{-1}$ and if $p > n$ and $L$ has full rank we have $L_{A}^{\dagger} = L^{\dagger}$. Hence, if $L^TL = I$, the standard- form is achieved by simply applying $L^T$. \\
In terms of the GSVD of $(A,L)$, the standard-form transformation has a much simpler form. The operator $L_{A}^{\dagger}$ can be written in terms of the GSVD as
\[
	L_{A}^{\dagger} = X^{-1}\Gamma^{\dagger}V^T,
\]
and hence \cref{eq:xM1} can be written as
\begin{equation}
	\bar{x}_{\mathcal{M}} = X^{-1}\Gamma^{\dagger}V^T\bar{y}, \quad \bar{y} = \argmin_{y}  \dfrac{1}{2}\Vert U\Sigma\Gamma^{\dagger}V^Ty - b\Vert_2^2 + \mathcal{R}(y).
\label{eq:xR}
\end{equation}
Similarly, \cref{eq:xN1} can be written in terms of the GSVD as
\begin{equation}
	\bar{x}_{\mathcal{N}} = X^{-1}\begin{bmatrix} 0 & 0 \\ 0 & I_{p-r_L} \end{bmatrix}U^Tb.
\label{eq:xN}
\end{equation}

\subsection{The SVD of $F_{\kappa}$}
In this section we derive the SVD of $F_{\kappa}$ in terms of the GSVD of $(A,L)$.
\begin{theorem}
	Let $F_{\kappa} = Y\Lambda Z^T$ be the SVD of $F_{\kappa}$. Let the GSVD of $\begin{bmatrix} A \\ L \end{bmatrix} = \begin{bmatrix} U\Sigma \\ V\Gamma \end{bmatrix}X$. Then
\begin{eqnarray*}
	Y & = & \begin{bmatrix} \kappa^{1/2}V\widetilde{\Sigma}_{\kappa,I}^{1/2} & \kappa V \widetilde{\Sigma}_{\kappa,I}^{-1/2} \Gamma\left( \Sigma^T\Sigma + \kappa\Gamma^T\Gamma\right)^{-1}\Sigma^T \\ \kappa U\Sigma\left( \Sigma^T\Sigma + \kappa\Gamma^T\Gamma\right)^{-1}\Gamma^T\widetilde{\Sigma}_{\kappa,I}^{-1/2} & -\kappa^{-1/2}U\widetilde{\Sigma}_{m,\kappa}^{1/2} \end{bmatrix} \\
	\Lambda & = & \begin{bmatrix} \widetilde{\Sigma}_{\kappa }^{1/2} \\ 0\end{bmatrix}\\
	Z & = & V,
\end{eqnarray*}
where $\widetilde{\Sigma}_{\kappa}= \kappa\left(I_p - \kappa\Gamma(\Sigma^T\Sigma + \kappa\Gamma^T\Gamma)^{-1}\Gamma^T\right)$, $\widetilde{\Sigma}_{\kappa,m} = \begin{bmatrix} \widetilde{\Sigma}_{\kappa } & 0 \\ 0 & I_{m-p}\end{bmatrix}$ if $m \geq n \geq p$ and $\widetilde{\Sigma}_{\kappa,m} = \widetilde{\Sigma}_{\kappa,I}$ if $m < n \leq p$, and the square root denotes the entry wise square root. If $p > n$ the diagonal matrix $\widetilde{\Sigma}_{\kappa}$ will have zeros on the diagonal. We denote $\widetilde{\Sigma}_{\kappa,I}$ to be the matrix $\widetilde{\Sigma}_{\kappa}$ where the zeros have been replaced by ones.
\end{theorem}
\begin{proof}
		Using the GSVD of $(A,L)$ we have $H_{\kappa }^{-1} = X^{-1}(\Sigma^T\Sigma + \kappa\Gamma^T\Gamma)^{-1}X^{-T}$ and hence $LH_{\kappa }^{-1}L^T = V\Gamma(\Sigma^T\Sigma + \kappa\Gamma^T\Gamma)^{-1}\Gamma^TV^T$. Given the fact that $V$ is orthonormal and $\Gamma(\Sigma^T\Sigma + \kappa\Gamma^T\Gamma)^{-1}\Gamma^T$ is a diagonal matrix the above expression is the SVD of $LH_{\kappa }^{-1}L^T$ and we obtain the expressions for $\Lambda$ and $Z$. To obtain $Y$, we first partition $Y = \begin{bmatrix} Y_{11} & Y_{12} \\ Y_{21} & Y_{22} \end{bmatrix}$. We have
		\begin{eqnarray*}
			F_{\kappa }F_{\kappa }^T  & = & Y\Lambda\Lambda^TY^T \\
		& \Longleftrightarrow & \begin{bmatrix} \kappa \left(I - \kappa  LH_{\kappa }^{-1}L^T\right)^2 & \kappa\sqrt{\kappa} \left(I - \kappa LH_{\kappa }^{-1}L^T\right)LH_{\kappa }^{-1}A^T \\
		 \kappa\sqrt{\kappa} AH_{\kappa }^{-1}L^T\left(I - \kappa LH_{\kappa}^{-1}L^T\right) & \kappa^{2} AH_{\kappa }^{-1}LL^TH_{\kappa }^{-1}A^T \end{bmatrix}  \\
		& = & \begin{bmatrix} Y_{11} & Y_{12} \\ Y_{21} & Y_{22} \end{bmatrix}\begin{bmatrix} \widetilde{\Sigma}_{\kappa} & 0 \\ 0 & 0 \end{bmatrix}\begin{bmatrix} Y_{11}^T & Y_{21}^T \\ Y_{12}^T & Y_{22}^T \end{bmatrix} = \begin{bmatrix}
		Y_{11}\widetilde{\Sigma}_{\kappa}Y_{11}^T & Y_{11}\widetilde{\Sigma}_{\kappa}Y_{21}^T  \\ Y_{21}\widetilde{\Sigma}_{\kappa}Y_{11}^T & Y_{21}\widetilde{\Sigma}_{\kappa}Y_{21}^T,
	\end{bmatrix}.
		\end{eqnarray*}
Plugging in the GSVD gives
\[
	F_{\kappa }F_{\kappa }^T = \begin{bmatrix}
		\kappa^{-1} V\widetilde{\Sigma}_{\kappa}^2V^T & \sqrt{\kappa} V\widetilde{\Sigma}_{\kappa}\Gamma(\Sigma^T\Sigma+\kappa\Gamma^T\Gamma)^{-1}\Sigma^TU^T \\
		\sqrt{\kappa} U\Sigma(\Sigma^T\Sigma+\kappa\Gamma^T\Gamma)^{-1}\Gamma^T\widetilde{\Sigma}_{\kappa}V^T & \kappa^{2} U\Sigma(\Sigma^T\Sigma+\kappa\Gamma^T\Gamma)^{-1}\Gamma^T\Gamma(\Sigma^T\Sigma+\kappa\Gamma^T\Gamma)^{-1}\Sigma^TU^T
	\end{bmatrix}.
\]
Solving for $Y_{11}$ gives:
\[
	Y_{11} = \kappa^{-1/2}V\widetilde{\Sigma}_{\kappa,I}^{1/2}.
\]
Using this in the upper right part gives:
\[
	Y_{21} = \kappa U\Sigma( \Sigma^T\Sigma + \kappa\Gamma^T\Gamma)^{-1}\Gamma^T\widetilde{\Sigma}_{\kappa,I}^{-1/2}.
\]
To solve for $Y_{12}$ and $Y_{22}$, we use
\[
	YY^T = \begin{bmatrix} Y_{11}Y_{11}^T + Y_{12}Y_{12}^T & Y_{11}Y_{21}^T + Y_{12}Y_{22}^T \\ Y_{21}Y_{11}^T + Y_{22}Y_{12}^T & Y_{21}Y_{21}^T + Y_{22}Y_{22}^T \end{bmatrix} = \begin{bmatrix} I_p & 0 \\ 0 & I_m \end{bmatrix}.
\]
The upper left part yields
\[
	Y_{12} = \kappa V \widetilde{\Sigma}_{\kappa,I}^{-1/2} \Gamma\left( \Sigma^T\Sigma+ \kappa\Gamma^T\Gamma\right)^{-1}\Sigma^T.
\]
The upper right part yields
\[
	Y_{22} = -\kappa^{-1/2}U\widetilde{\Sigma}_{\kappa,m}^{1/2}.
\]
\end{proof}
Note that the singular values are ordered in ascending order. We have the following corollary.
\begin{corollary}
	If $m\geq n$ and $p<n$ the singular values of $F_{\kappa}$ are given by
	\[
		\psi_i(F_{\kappa}) = \sqrt{\dfrac{\sigma_{n-i+1}^2}{\sigma_{n-i+1}^2/\kappa + \gamma_{n-i+1}^2}}.
	\]
	If $m < n$ and $p > n$ the singular values of $F_{\kappa}$ are given by
	\[
		\psi_i(F_{\kappa}) = \begin{cases} \sqrt{\kappa} & \text{ if } i \leq p - r_{L} \\ 
		\sqrt{\dfrac{\sigma_{m-i+1}^2}{\sigma_{m-i+1}^2/\kappa + \gamma_{m-i+1}^2}} & \text{ if } p - r_{L} < i \leq p - r_{L} + r_A \\  
		0 & \text{ if } i > p - r_{L} + r_A	
		\end{cases}
	\]
\end{corollary}
The question arises whether there is a direct relation between the singular values of $A$ and the $\sigma_i$. The answer is no, but we do, however, have the following result from \cite{Hansen1989}:
\begin{theorem}[{\cite[Thm. 2.4]{Hansen1989}}]
	Let $\psi_i(A)$ and $\psi_i(L)$ denote the singular values of $A$ and $L$ respectively and let $\sigma_i$ and $\gamma_i$ denote the nonzero entries of the matrices $\Sigma$ and $\Gamma$ respectively. Then for all $\sigma_i,\gamma_i \neq 0$
	\begin{eqnarray*}
		\left\Vert \begin{bmatrix} A\\L\end{bmatrix}^{\dagger}\right\Vert^{-1}_2 \leq & \dfrac{\psi_{r-i+1}(A)}{\sigma_i} & \leq  \left\Vert \begin{bmatrix} A\\L\end{bmatrix}\right\Vert_2, \\
		\left\Vert \begin{bmatrix} A\\L\end{bmatrix}^{\dagger}\right\Vert^{-1}_2 \leq & \dfrac{\psi_{i}(L)}{\gamma_i} & \leq  \left\Vert \begin{bmatrix} A\\L\end{bmatrix}\right\Vert_2.
	\end{eqnarray*}
\end{theorem}
\begin{remark}
This result shows that, if the operator $A$ has quickly decaying singular values, the $\sigma_i$ will have the same behavior, see also \cite[p. 24]{Hansen1998}. This is an important result because it shows how the ill-conditioning of $A$ transfers over to $F_{\kappa}$. Note that if $\sigma_{i} \approx 0$ we have $\gamma_i \approx 1$ and the singular values of $\psi_i(F_{\kappa }) = \sqrt{\dfrac{\sigma_{r-i+1}}{\sigma_{r-i+1}/\kappa + \gamma_{r-i+1}}} \approx \sqrt{\dfrac{\sigma_{r-i+1}}{\sigma_{r-i+1}/\kappa + 1}} \approx 0$. Hence, if the operator $A$ is severely ill-posed, this ill-posedness is inherited by the operator $F_{\kappa}$.
\end{remark}

\subsection{Limiting cases}
\subsubsection{The limit $\kappa\to\infty$ if $p < n$} If $L = I$ the limit $\kappa\to \infty$ yields the original optimization problem. However, if $L\neq I$, it is not immediately clear what happens in the limit $\kappa\to \infty$ due to the presence of the operator $L$. In this section we derive this limit using the GSVD of $(A,L)$. We will show that the in the limit $\kappa\to\infty$ SR3 applies a standard-form transformation. We will proceed as follows. Recall that the variable $x$ in SR3 is given by
\begin{equation}
	\bar{x}_{\kappa} = H_{\kappa}^{-1}\left( \kappa L^T\bar{y}_{\kappa} + A^Tb\right) = \kappa H_{\kappa}^{-1} L^T\bar{y}_{\kappa} + H_{\kappa}^{-1} A^Tb := x_1 + x_2,
\label{eq:SR3_x1x2}
\end{equation}
consisting of the two parts $x_1$ and $x_2$.
We will now show that, in the limit $\kappa\to\infty$, SR3 applies a standard-form transformation, by showing that $x_1$ and $x_2$ defined in \cref{eq:SR3_x1x2} satisfy
\begin{equation}
	x_1  = \bar{x}_{\mathcal{M}}, \quad x_2 = \bar{x}_{\mathcal{N}} ,
\end{equation}
where $x_{\mathcal{M}}$ and $x_{\mathcal{N}}$ are determined by the standard-form transformation, given by \cref{eq:xR} and \cref{eq:xN} respectively.
\\
Given the GSVD of $(A,L)$, the matrix $F_{\kappa}$ and the vector $g_{\kappa}$ are given by
\begin{equation}
	F_{\kappa} = \begin{bmatrix} \sqrt{\kappa} V\left( I_p - \kappa\Gamma\left( \Sigma^T\Sigma + \kappa\Gamma^T\Gamma\right)^{-1}\Gamma^T\right)V^T \\ \kappa U\Sigma\left(\Sigma^T\Sigma + \kappa\Gamma^T\Gamma\right)^{-1}\Gamma^TV^T \end{bmatrix},
\label{eq:Fkgsvd}
\end{equation}
and
\begin{equation}
	g_{\kappa} = \begin{bmatrix} \sqrt{\kappa}V\Gamma(\Sigma^T\Sigma + \kappa\Gamma^T\Gamma)^{-1}\Sigma^TU^Tb \\ U\left(I_m - \Sigma\left(\Sigma^T\Sigma + \kappa\Gamma^T\Gamma\right)^{-1}\Sigma^T\right)U^Tb\end{bmatrix}.
\label{eq:gkgsvd}
\end{equation}
As $\kappa\to \infty$ we have
\[
	F_{\kappa} \to \begin{bmatrix} 0 \\ U\Sigma\Gamma^{\dagger}V^T \end{bmatrix} \text{ and } g_{\kappa} \to \begin{bmatrix} 0 \\ b \end{bmatrix}.
\]
Hence, as $\kappa\to\infty$, we obtain
\begin{equation}
	\bar{y}_{\kappa} = \argmin_{y} \dfrac{1}{2}\Vert U\Sigma\Gamma^{\dagger}V^Ty - b\Vert_2^2 + \mathcal{R}(y).
\label{eq:SR3_std1}
\end{equation}
Using the GSVD, we have
\[
	H_{\kappa}^{-1} = X^{-1}\left( \Sigma^T\Sigma + \kappa\Gamma^T\Gamma \right)^{-1}X^{-T},
\]
 and hence as $\kappa\to\infty$ we have
\[
	\left( \Sigma^T\Sigma + \kappa\Gamma^T\Gamma \right)^{-1} \to
	\begin{bmatrix}
	0 & 0 \\
	0 & I_{p-r_L}
	\end{bmatrix}.
\]
Hence,
\begin{equation}
	H_{\kappa}^{-1}\to X^{-1}	\begin{bmatrix}
	0 & 0 \\
	0 & I_{p-r_{L}}
	\end{bmatrix}X^{-T}.
\end{equation}
Recall that the last columns of $X$ are a basis for the nullspace of $L$ and hence $H_{\kappa}$ projects onto the nullspace of $L$. Using the GSVD of $(A,L)$ we see that
\[
	\lim_{\kappa\to\infty} x_1 := \lim_{\kappa\to\infty} H_{\kappa}^{-1}A^Tb = X^{-1}	\begin{bmatrix}
	0 & 0 \\
	0 & I_{p-r_{L}}
	\end{bmatrix}U^Tb,
\]
which is equivalent to the nullspace component from \eqref{eq:xN}. \\\\
We now show that $x_1$ corresponds to the part in the range of $L$. We have
\[
	x_1 := \kappa H_{\kappa}^{-1}L^T\bar{y} = \kappa X^{-1}\left(\Sigma^T\Sigma + \kappa\Gamma^T\Gamma \right)^{-1}\Gamma^TV^T\bar{y}.
\]
The elements of the diagonal matrix $\kappa\left(\Sigma^T\Sigma + \kappa\Gamma^T\Gamma \right)^{-1}\Gamma^T$ are
\[
	\begin{array}{cc} \dfrac{\gamma_i}{\sigma_i^2/\kappa + \gamma_i^2} & \text{ if } i \leq r_{L} \\ 0 & \text{ if } i > r_{L} \end{array},
\]
and as $\kappa\to\infty$
\[
	\begin{array}{cc} \dfrac{1}{\gamma_i} & \text{ if } i \leq r_{L} \\ 0 & \text{ if } i > r_{L} \end{array}.
\]
Hence, as $\kappa\to\infty$
\[
	\kappa \left(\Sigma^T\Sigma + \kappa\Gamma^T\Gamma \right)^{-1}\Gamma^T \to \Gamma^{\dagger},
\]
and thus
\[
	\kappa H_{\kappa}^{-1}L^T \to X^{-1}\Gamma^{\dagger}V^T =  L_{A}^{\dagger}.
\]
The limit for the component $x_1$ is now given by
\[
	\lim_{\kappa\to \infty} x_1 = X^{-1}\Gamma^{\dagger}V^T\bar{y}_{\kappa} = L_{A}^{\dagger}\bar{y}_{\kappa},
\]
where $\bar{y}_{\kappa}$ solves
\[
	\bar{y}_{\kappa} = \argmin_{y} \dfrac{1}{2} \Vert U\Sigma\Gamma^{\dagger}V^Ty - b\Vert_2^2 + \mathcal{R}(y),
\]
which is equivalent to \eqref{eq:xR}. 

\subsubsection{The limit $\kappa\to\infty$ if $p > n$}
If $p > n$, the limit $\kappa\to\infty$ is a bit more subtle. For large $\kappa$, we have
\begin{equation}
	F_{\kappa} \sim \begin{bmatrix} 
	\multicolumn{2}{c}{V\begin{bmatrix} 0_{r_L\times r_L} & 0 \\ 0 & \sqrt{\kappa}I_{p-r_L} \\ \end{bmatrix} V^T} \\ 
	\multicolumn{2}{c}{U\Sigma\Gamma^{\dagger}V^T} 
	\end{bmatrix}, \quad \text{ and } g_{\kappa} \sim \begin{bmatrix} 0_{p\times 1} \\ b \end{bmatrix}
\label{eq:Fp>n}
\end{equation}
Hence, for large $\kappa$, SR3 solves a system of the form 
\[
	\begin{bmatrix} 
	\multicolumn{2}{c}{\sqrt{\kappa}V_{p-r_L}V_{p-r_L}^T} \\ 
	\multicolumn{2}{c}{U\Sigma\Gamma^{\dagger}V^T} 
	\end{bmatrix}y = \begin{bmatrix} 0_{p\times 1} \\ b \end{bmatrix},
\]
where $V_{p-r_L}$ are the \textit{last} $p-r_L$ columns of $V$, which means that $V_{p-r_L}V_{p-r_L}^T = \mathcal{P}_{\mathcal{N}(L^T)}$. Because  $V_{p-r_L}V_{p-r_L}^Ty = 0$, the solution has no parts in $\mathcal{N}(L^T)$, and is restricted to the subspace $\mathcal{R}(L)$. The bottom part of $F_{\kappa}$ is equal to the case $p < n$, and hence corresponds to matrix $AL_{A}^{\dagger}$. Let $\bar{y}_{\text{std}}$ be the solution to the standard-form transformed system. Then, as $\kappa\to\infty$, the minimizer $\bar{y}_{\kappa}$ of SR3 satisfies
\begin{equation}
	\bar{y}_{\text{std}}= \mathcal{P}_{\mathcal{R}(L)}\bar{y}_{\kappa}.
\label{eq:p>n_condition}
\end{equation}
However, looking at the original formulation in \cref{eq:relaxed}, we see that as $\kappa\to\infty$ we have
\[
	y = Lx,
\]
which means that $y\in\mathcal{R}(L)$. Hence, condition \cref{eq:p>n_condition} is immediately satisfied and the solutions are the same.
 
\subsection{The limit $\kappa\to 0$} 
The limit $\kappa\to 0$ is much easier to derive. Recall that
\[
	\bar{x}_{\kappa} = H_{\kappa}^{-1}\left(A^Tb + \kappa L^T\bar{y}_{\kappa}\right).
\]
As $\kappa\to 0$ we have $\kappa H_{\kappa}^{-1}L^T\bar{y}_{\kappa}\to 0$ and $H_{\kappa}\to (A^TA)^{-1}$. Hence $\lim_{\kappa\to 0}x_{\kappa} = (A^TA)^{-1}A^Tb$ which is the unregularized minimum norm solution. 

\subsection{Relation to the standard-form transformation}
\subsubsection{The case $p \leq n$}
We have shown that as $\kappa\to \infty$ SR3 implicitly applies a standard-form transformation and that as $\kappa\to 0$ the system is unregularized. The question arises what happens for finite $\kappa > 0$. To show what happens, we rewrite the singular values of $F_{\kappa}$ as
\[
	\psi_i(F_{\kappa}) = \sqrt{\dfrac{\sigma_{r-i+1}^2}{\sigma_{r-i+1}^2/\kappa + \gamma_{r-i+1}^2}} = \sqrt{\dfrac{\sigma_{r-i+1}^2/\gamma_{r-i+1}^2}{\frac{\sigma_{r-i+1}^2/\gamma_{r-i+1}^2}{\kappa} + 1}} = \sqrt{\dfrac{\psi_i^2\left(AL_{A}^{\dagger}\right)}{\psi_i^2\left(AL_{A}^{\dagger}\right)/\kappa + 1}}.
\]
This is equivalent to equation 9 in \cite{Zheng2019}, where it was shown that if $L^TL = I$,
\[
	\psi_i(F_{\kappa}) = \dfrac{\psi_i^2(A)}{\psi_i^2(A)/\kappa + 1}.
\]
This shows that SR3 is applied to the matrix $AL_{A}^{\dagger}$. This leads to the following theorem.
\begin{theorem}
Let $p \leq n$. The following diagram commutes.

\begin{equation*}
\begin{tikzcd}[ampersand replacement=\&]
\min_{x}\frac{1}{2}\Vert Ax - b\Vert_2^2 + \lambda\mathcal{R}(Lx) \arrow[r, "\text{SR3}"] \arrow["",d]
\& 
\begin{array}{@{}r@{\,}l}  
& \bar{y}_{\kappa} = \argmin_{y}\textstyle{\frac{1}{2}}\Vert F_{\kappa}y - g_{\kappa}\Vert_2^2 + \lambda\mathcal{R}(y) \\
& \bar{x}_{\kappa} = H_{\kappa}^{-1}(\kappa L^T\bar{y}_{\kappa} + A^Tb) \arrow[shift right, swap]{d}{} 
\end{array}
\& 
\\
\begin{array}{@{}r@{\,}l}  
& \bar{z}_{\kappa} = \argmin_{z} \frac{1}{2}\Vert AL_{A}^{\dagger}z - b\Vert_2^2 + \lambda\mathcal{R}(z) \\
& \bar{x}_{\kappa} = L_{A}^{\dagger}\bar{z}_{\kappa} + x_{\mathcal{N}} 
\end{array}
\arrow[r,"\text{SR3}"] 
\& 
\begin{array}{@{}r@{\,}l}  
& \bar{y}_{\kappa} = \argmin_{y}\frac{1}{2}\Vert F_{\kappa}y - g_{\kappa}\Vert_2^2 + \lambda\mathcal{R}(y) \\
& \bar{z}_{\kappa} = H_{\kappa}^{-1}(\kappa \bar{y}_{\kappa} + (AL_A^{\dagger})^Tb) \\
& \bar{x}_{\kappa} = L_{A}^{\dagger}\bar{z}_{\kappa} + x_{\mathcal{N}} \arrow[shift right, swap]{u}{}
\end{array} 
\end{tikzcd}
\end{equation*}
\end{theorem}
%

\subsubsection{The case $p > n$}
If $p > n$ the situation is different. Recall that the singular values of $F_{\kappa}$ are given by
\[
	\psi_i(F_{\kappa}) = \begin{cases} \sqrt{\kappa} & \text{ if } i \leq p - r_{L} \\ 
	\sqrt{\dfrac{\psi_i^2\left(AL_{A}^{\dagger}\right)}{\psi_i^2\left(AL_{A}^{\dagger}\right)/\kappa + 1}}. & \text{ if } p - r_{L} < i \leq p - r_{L} + r_A \\  
	0 & \text{ if } i > p - r_{L} + r_A	
	\end{cases}
\]
The singular values for $F_{\kappa}$ when SR3 is applied to $AL_{A}^{\dagger}$ are given by
\[
	\psi_i(F_{\kappa}) = 
	\begin{cases} \sqrt{\dfrac{\psi_i^2\left(AL_{A}^{\dagger}\right)}{\psi_i^2\left(AL_{A}^{\dagger}\right)/\kappa + 1}}. & \text{ if } i \leq r_A \\  
	0 & \text{ if } i > r_A	
	\end{cases}
\]
Hence, there are extra singular values $\sqrt{\kappa}$ when SR3 is applied to the general-form system as opposed to the standard-form system. The difference may be seen from the expression \cref{eq:Fkgsvd}. We have
\[
	\kappa\Gamma\left(\Sigma^T\Sigma + \kappa\Gamma^T\Gamma\right)\Gamma^T = 
	\begin{cases}
		\begin{bmatrix} I_{n-r_A} & 0 & 0 \\ 0 & \kappa\Gamma_m(\Sigma^T\Sigma + \kappa\Gamma_m^T\Gamma_m)\Gamma_m^T & 0 \\ 0 & 0 & 0\end{bmatrix} & \text{ if }  p  > n \\
		\begin{bmatrix} I_{p-r_A} & 0 \\ 0 & \kappa\Gamma_m(\Sigma^T\Sigma + \kappa\Gamma_m^T\Gamma_m)\Gamma_m^T \end{bmatrix} & \text{ if }  p  \leq n 
	\end{cases}
\]
Hence, the top part of $F_{\kappa}$ is different. Before we state our theorem let us introduce some notation. For the general-form problem, let the function $\varphi$ be defined as the spectral cut-off function that makes the first $p - r_L$ singular values of $F_{\kappa}$ zero. Similarly, for the standard-form transformed problem, let $\varrho$ be defined as the function that makes $p - r_L$ singular values that are 0 equal to $\sqrt{\kappa}$ and accordingly permutes the SVD. We then have $\varphi\circ\varrho = \text{Id}$. We have the following theorem.
\begin{theorem}
Let $p > n$. The following diagram commutes.

\begin{equation*}
\begin{tikzcd}[ampersand replacement=\&]
\min_{x}\frac{1}{2}\Vert Ax - b\Vert_2^2 + \lambda\mathcal{R}(Lx) \arrow[r, "\text{SR3}"] \arrow["",d]
\& 
\begin{array}{@{}r@{\,}l}  
& \bar{y}_{\kappa} = \argmin_{y}\textstyle{\frac{1}{2}}\Vert F_{\kappa}y - g_{\kappa}\Vert_2^2 + \lambda\mathcal{R}(y) \\
& \bar{x}_{\kappa} = H_{\kappa}^{-1}(\kappa L^T\bar{y}_{\kappa} + A^Tb) \arrow[shift right, swap]{d}{\varphi} 
\end{array}
\\
\begin{array}{@{}r@{\,}l}  
& \bar{z}_{\kappa} = \argmin_{z} \frac{1}{2}\Vert AL_{A}^{\dagger}z - b\Vert_2^2 + \lambda\mathcal{R}(y) \\
& \bar{x}_{\kappa} = L_{A}^{\dagger}\bar{z}_{\kappa}  + x_{\mathcal{N}} 
\end{array}
\arrow[r,"\text{SR3}"] 
\& 
\begin{array}{@{}r@{\,}l}  
& \tilde{\bar{y}}_{\kappa} = \argmin_{y}\frac{1}{2}\Vert \tilde{F}_{\kappa}y - g_{\kappa}\Vert_2^2 + \lambda\mathcal{R}(y) \\
& \tilde{\bar{z}}_{\kappa}  = H_{\kappa}^{-1}(\kappa \tilde{\bar{y}}_{\kappa}  + (AL_A^{\dagger})^Tb) \\
& \tilde{\bar{x}}_{\kappa}  = L_{A}^{\dagger}\bar{z}_{\kappa}  + x_{\mathcal{N}} \arrow[shift right, swap]{u}{\varrho}
\end{array}
\end{tikzcd}
\end{equation*}
\end{theorem}

\section{Approximating the value function}
\label{sec:value_function}
In this section we quantify the distance between the Pareto curve of the original problem and the Pareto curve of the relaxed problem in terms of $\kappa$. We first describe the value function of the problem and then present our theorem.

The value function of an optimization problem expresses the value of the objective at the solution as a function of the other parameters. Using the standard-form transformation, we can, without loss of generality, consider the standard-form value function:
\[
	\phi_\kappa(\tau) = \min_{y} \|F_{\kappa}y - g_{\kappa}\|_2 \quad \text{s.t.} \quad \Vert y\Vert_p \leq \tau.
\]
\subsection{Value function for $\kappa \rightarrow \infty$}
We have seen that for $\kappa\rightarrow\infty$, we retrieve the unrelaxed problem with value function
\[
	\phi_\infty(\tau) = \min_{y} \|Ay - b\|_2 \quad \text{s.t.} \quad \Vert y\Vert_p \leq \tau.
\]
Following \cite{VanDenBerg2008} we obtain the following (computable) upper and lower bounds for the value function
\[
b^T\widetilde{r} - \tau \|A^T\widetilde{r}\|_q \leq \phi_\infty(\tau) \leq \|\widetilde{r}\|_2,
\]
where $\widetilde{y}$ is any feasible point (i.e., $\|\widetilde{y}\|_p \leq \tau$), and $\widetilde{r} = b - A\widetilde{y}$ is the corresponding residual and $p^{-1} + q^{-1} = 1$.
Moreover, by {\cite[Col. 2.2]{VanDenBerg2008}} the derivative of the value function is given by
\[
\phi_\infty'(\tau) = - \|A^T\overline{r}\|_q/\|\overline{r}\|_2,
\]
with $\overline{r} = b - A\overline{y}$ and $\overline{y}=\argmin_{\|y\|_p\leq\tau} \|Ay - b\|_2$.

To gain some insight in the behaviour of the value function, we consider $\phi_\infty$ and $\phi_\infty'$ at $\tau = 0$ and $\tau = \tau_{*} = \|A^{\dagger}b\|_p$:
\[
\phi_\infty(0) = \|b\|_2 , \quad \phi_\infty'(0) = -\|A^Tb\|_q/\|b\|_2,
\]
\[
\phi_\infty(\tau_*) = \|(I - AA^{\dagger})b\|_2, \quad \phi_\infty'(\tau_*) = 0.
\]
This immediately suggests that $\phi_\infty$ decreases linearly near $\tau = 0$ (the zero solution) and flattens of near $\tau = \tau_*$ (the unconstrained minimizer). Since $\phi_\infty$ is known to be convex, its second derivative is always positive and will gradually bend the curve from decreasing to flat. How fast this happens and whether one can expect the typical L-shape, depends on how fast the curve decreases initially. We can bound $\phi_\infty'(0)$ as follows. We let $b = Ay$ and find
\[
\|A^Tb\|_q = \|A^T\!Ay\|_q \geq C_q \|A^T\! Ay\|_2 \geq C_q\|A^\dagger\|^2 \|y\|_2,
\]
where $C_q$ is a constant that exists due to the equivalence of norms. Furthermore,
\[
\|b\|_2 = \|Ay\|_2 \leq \|A\|_2 \|y\|_2.
\]
From this we get
\[
\phi'_\infty(0) \leq - C_q \kappa_2(A) \|A^\dagger\|_2,
\]
with $\kappa_2(A) = \|A\|_2\|A^\dagger\|_2$ the condition number of $A$. We thus expect a steep slope for ill-conditioned problems, giving rise for the characteristic $L$-shape of the curve. While this behavior is well-established for $p=2$ where it can be analysed using the SVD of $A$ \cite{Hansen1992}, this analysis gives us new insight in the behavior of the Pareto curve for ill-posed problems for general $p$. An example for $p = 1$, $L = I$ is shown in figure \ref{fig:pareto_l1}.
\begin{figure}
\centering
\includegraphics[scale=.5]{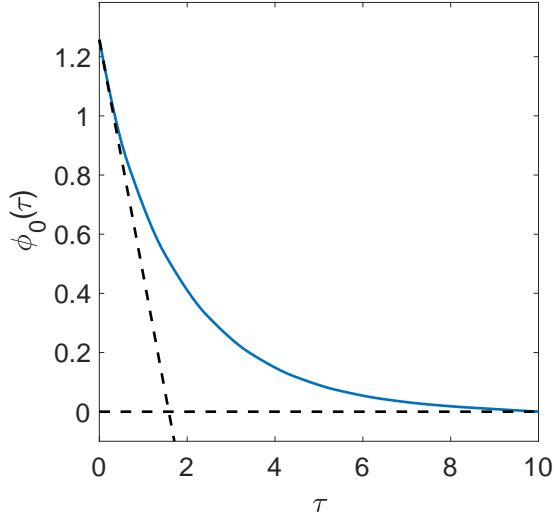}
\caption{Pareto curve for an ill-posed problem; the matrix $A$ is diagonal with elements $e^{-(i - 1)/2}$ for $i = 1,2, \ldots 10$; $b = Ax$ with $x = (1, 1, \ldots, 1)$. The tangent lines at $\tau = 0$ and $\tau = \tau_*$ are shown in black.}
\label{fig:pareto_l1}
\end{figure}

\subsection{Relaxed value function}
We now present our theorem on the distance between the Pareto curve of the original problem and the Pareto curve of the relaxed problem.
\begin{theorem}
\label{thm:valuefunction}
	The distance between the Pareto curve of the original problem and the Pareto curve of the relaxed problem is given by
	\[
		\left(\phi_\kappa(\tau)\right)^2 - \left(\phi_\infty(\tau)\right)^2 = -\kappa^{-1} \Vert A^T(b - A\overline{y}_{\kappa})\Vert_2^2 + \mathcal{O}\left(\kappa^{-2}\right),
	\]
	where $\overline{y}_{\kappa}$ is the solution of the relaxed problem. In particular, we have
	\[
		\phi_{\kappa}(\tau) \leq \phi_{\infty}(\tau).
	\]
\end{theorem}
\begin{proof}
	Let $\epsilon = \kappa^{-1}$. The relaxed value function can be expressed as
\[
	\phi_{\epsilon}(\tau) = \min_y \|F_{\epsilon}y - g_{\epsilon}\|_2 \quad \text{s.t.} \quad \|y\|_p \leq \tau.
\]
For $\epsilon < \|A\|^2_2$ we can expand $H_{\epsilon}^{-1} = \epsilon I - \epsilon^2 A^TA + \mathcal{O}(\epsilon^{3})$ and get
\[
	F_{\epsilon} = \left(\begin{matrix}A - \epsilon AA^TA + \mathcal{O}(\epsilon^2) \\ \epsilon^{1/2}A^TA + \mathcal{O}(\epsilon^{3/2})\end{matrix}\right), \quad
g_{\epsilon} = \left(\begin{matrix}b - \epsilon A^Tb + \mathcal{O}(\epsilon^{2}) \\ \epsilon^{-1/2}A^Tb + \mathcal{O}(\epsilon^{3/2})\end{matrix}\right).
\]
Introduce
\[
	f(\epsilon) = \left(\phi_\epsilon(\tau)\right)^2 = \min_{x,y} \|Ax -  b\|_2^2 + \epsilon^{-1}\|x - y\|_2^2 \quad \text{s.t.}\quad \|y\|_{p} \leq \tau.
\]
We have $f(0) = \min_{\|y\|_{p} \leq \tau} \|Ay - b\|_2^2 = \left(\phi_0(\tau)\right)^2$. Furthermore
\[
	f'(\epsilon) = -\epsilon^{-2}\|\overline{x}_{\epsilon} - \overline{y}_{\epsilon}\|_2^2,
\]
where $\overline{x}_{\epsilon} = H_{\epsilon}^{-1}(A^Tb + \epsilon^{-1}\overline{y}_{\epsilon})$ and $\overline{y}_{\epsilon}$ is the optimal $y$.
With this we find
\begin{equation}
	\left(\phi_\epsilon(\tau)\right)^2 - \left(\phi_0(\tau)\right)^2 = \epsilon f'(\eta) = -\epsilon\eta^{-2} \|\overline{x}_{\eta} - \overline{y}_{\eta}\|_2^2.
\end{equation}
We conclude that $\phi_{\epsilon}(\tau) \leq \phi_0(\tau)$. Alternatively, we can express
\begin{equation}
	\left(\phi_\epsilon(\tau)\right)^2 - \left(\phi_0(\tau)\right)^2 = -\epsilon^{-1} \|\overline{x}_{\epsilon} - \overline{y}_{\epsilon}\|_2^2 + \mathcal{O}(\epsilon^2).
\label{eq:pareto_distance}
\end{equation}
For small $\epsilon$ we get
\[
	f'(\epsilon) = -\|A^T(b - A\overline{y}_{\epsilon})\|_2^2 + \mathcal{O}(\epsilon).
\]
Plugging this expression into \cref{eq:pareto_distance} gives the desired result.
\end{proof}

\begin{remark}
\cref{thm:valuefunction} can be used to explain the behaviour of the Pareto curves observed in the examples in section \ref{motivating_examples}:
\begin{itemize}
\item The error gets smaller for large $\tau$. For an unconstrained problem we have ${\Vert A^T(b - A\overline{y}_{\kappa})\Vert_2 = 0}$ as $\kappa \rightarrow\infty$. An example is shown in \cref{fig:pareto2_l1}.
\item The elbow of the Pareto curves coincide; $\phi_\infty$ decreases fast initially for ill-posed problems (cf. \cref{fig:pareto_l1}) while $\phi_{\kappa}$ decreases less fast due to the implicit regularizating effect of the relaxation. Since $0 \leq \phi_\kappa \leq \phi_\infty$, the relaxed Pareto curve is pushed down and is therefore likely to have the elbow at the same location as $\phi_\infty$.
\end{itemize}

\end{remark}

\begin{figure}
\centering
\includegraphics[scale=.4]{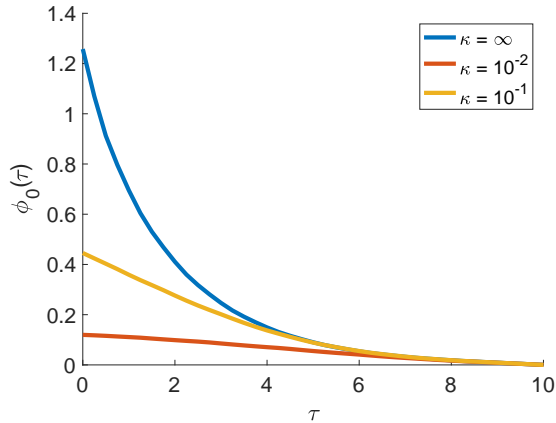}
\caption{Pareto curve for an ill-posed problem; the matrix $A$ is diagonal with elements $e^{-(i - 1)/2}$ for $i = 1,2, \ldots 10$; $b = Ax$ with $x = (1, 1, \ldots, 1)$. The approximations for various values of $\epsilon$ are shown as well.}
\label{fig:pareto2_l1}
\end{figure}

\section{Implementation}
\label{sec:implementation}
Recall from the introduction that we implement SR3 as follows:
\begin{eqnarray}
	x_{k+1} &\leftarrow& \left(A^T\!A + \kappa L^T\!L\right)^{-1} \left(A^Tb + \kappa L^Ty_k\right) \\
	y_{k+1} &\leftarrow& \text{prox}_{1/\kappa \mathcal{R}} \left(Lx_{k}\right).
\end{eqnarray}
The last equation shows that for the choice $\mathcal{R}(\cdot) = \lambda\Vert\cdot\Vert_p^p$ there is a relation between the parameters $\kappa$ and $\lambda$. More specifically, $\lambda$ depends on $\kappa$ and hence we write $\lambda(\kappa)$. Given the optimal $\lambda_{\star}$, we have $\lambda(\kappa) = \lambda_{\star}\cdot\kappa$. Note that if we use the constrained formulation \cref{eq:lstau}, the dependence on the stepsize is lost because the proximal operator is the indicator function, and there is no relation between $\tau$ and $\kappa$. \\
The computational bottleneck is in the first step, which is the solution to the large-scale linear system
\begin{equation}
	\left( A^T\!A + \kappa L^T\!L\right)x_{k} = A^Tb + \kappa L^Ty_{k-1}.
\label{eq:inner_system}
\end{equation}
To avoid explicitly forming $A^T\!A$ and $L^T\!L$, we instead solve the following minimization problem
\begin{equation}
	\min_{x} \left\Vert\begin{bmatrix} A \\ \sqrt{\kappa} L\end{bmatrix}x - \begin{bmatrix} b \\ \sqrt{\kappa}y_{k-1} \end{bmatrix}\right\Vert_2^2,
\label{eq:ls_system}
\end{equation}
with LSQR. \\
 We will numerically investigate how only partially solving \cref{eq:ls_system} affects the convergence of SR3. This has been investigated for ADMM in \cite{Eckstein2017,Eckstein2018,MarquesAlves2020}. The convergence of FISTA with an inexact gradient has been analyzed in \cite{Schmidt2011}. The key message is that the error has to go down as the iterations increase. \\
In our implementation, we propose two extra ingredients to make SR3 suitable for large-scale problems: warm starts and inexact solves of \eqref{eq:ls_system}. Both ingredients are also used in the implementation of ADMM \cite{Boyd2011}. However, we propose a new stopping criterion for the inexact solves of \eqref{eq:ls_system}. \\
A \textit{warm start} is a technique used in inner-outer schemes, where the solution of the previous inner iteration serves as an initial guess to the new inner iteration. That is, we solve
\begin{equation}
	\min_{x} \left\Vert\begin{bmatrix} A \\ \sqrt{\kappa} L\end{bmatrix}x - \left(\begin{bmatrix} b \\ \sqrt{\kappa}y_{k-1} \end{bmatrix} - \begin{bmatrix} A \\ \sqrt{\kappa} L\end{bmatrix}x_{k-1} \right) \right\Vert_2^2.
\label{eq:ls_system_warm_start}
\end{equation}
By \textit{inexact solves} we mean finding an approximate solution to \eqref{eq:ls_system_warm_start}. The level of inexactness is determined by the difference between the true solution and the inexact solution. There are various ways in which one can solve the optimization problem inexactly. One way is to simply determine a maximum number of iterations. However, the number of iterations to solve \eqref{eq:ls_system} can vary strongly per outer iteration. Moreover, we may not want to solve the inner system with high precision in the first few outer iterations, because this does not result in significant improvement in the next outer iteration. Recently, the authors in \cite{vanLeeuwen2020nonsmooth} proposed a criterion to determine the amount of inexactness for inner-outer schemes. The idea is to stop the inner iteration once the difference in the resulting outer iterate becomes stagnant. Let $x_k$ denote the current inner iterate and $y_k = \text{prox}_{1/\kappa \mathcal{R}} \left(Lx_k\right)$ the resulting outer iterate by applying the proximal operator. Then the authors in \cite{vanLeeuwen2020nonsmooth} propose to stop the inner iterations if
\begin{equation}
	\Vert x_{k+1} - x_k \Vert < \rho \Vert y_{k+1} - y_k\Vert,
\end{equation}
for some user defined constant $\rho$. We propose a similar criterion, namely to stop if
\begin{equation}
	\dfrac{\Vert y_{k+1} - y_k\Vert}{\Vert y_k\Vert} < \epsilon,
\end{equation}
for some user defined threshold $\epsilon$. The index $k$ refers to the iteration of the iterative method applied to the inner iteration. This yields the proposed implementation of SR3, shown in \cref{alg:SR3}. Note that in line 4 of the algorithm we use the LSQR algorithm, and we build on the Krylov subspace from the previous step.
\begin{algorithm}
	\caption{Implementation of SR3}
	\label{alg:SR3_algorithm}
	\begin{algorithmic}[1]
		\Require Operators $A$ and $L$, the data $b$ and the parameters $\kappa$, $\lambda$ and $\epsilon$.
		\Ensure Approximate solution $x_k$.
		\While{$\Vert x_{k+1} - x_k\Vert > \delta$}
			\State $l = 0$.
			\While{$\dfrac{\Vert \tilde{y}_{l+1} - \tilde{y}_{l}\Vert}{\Vert \tilde{y}_l\Vert} > \epsilon$} \Comment{Run LSQR. We do \textbf{not} restart LSQR every iteration!}
				\State $x_l = \argmin_{x\in\mathcal{K}_l((A,\sqrt{\kappa}L),(b-Ax_k,\sqrt{\kappa}(y_0 - Lx_k)))}\left\Vert\begin{bmatrix} A \\ \sqrt{\kappa} L\end{bmatrix}x - \left(\begin{bmatrix} b \\ \sqrt{\kappa}y_{k} \end{bmatrix} - \begin{bmatrix} A \\ \sqrt{\kappa} L\end{bmatrix}x_{k} \right) \right\Vert_2^2$.
				\State $\tilde{y}_{l+1} = \text{prox}_{1/\kappa \mathcal{R}} \left(Lx_l\right)$. \Comment{Prospective update}
				\State $l = l + 1.$
			\EndWhile
			\State $y_k = \tilde{y}_{l+1}$.
			\vspace{1mm}
			\State $k = k + 1.$
 		\EndWhile
	\end{algorithmic}
\label{alg:SR3}
\end{algorithm} \\

It is important to note that the influence of $\kappa$ on the outer iteration is different from the influence of $\kappa$ on the inner iteration. The improved conditioning of the matrix $F_{\kappa}$ pertains to the convergence of the outer iteration. The convergence of the inner iteration is completely determined by the properties of the matrix $H_{\kappa}^{-1}$. It is important to note that using the GSVD of $(A,L)$ we get
\[
	H_{\kappa} = A^TA + \kappa L^TL = X^T\left(\Sigma^T\Sigma + \kappa \Gamma^T\Gamma\right)X,
\]
\textbf{but this is not the SVD of $H_{\kappa}$}, because $X$ is not orthonormal. Therefore, the matrix $\Sigma^T\Sigma + \kappa \Gamma^T\Gamma$ does not tell us anything about the convergence rate when solving linear systems involving $H_{\kappa}$.

\section{Numerical experiments}
\label{sec:experiments}
In this section we verify the results from \cref{sec:analysis} numerically. Furthermore, we implement \cref{alg:SR3} and test it on two examples. We use two examples that are regularized by TV regularization, which we solve in its constrained form, i.e.
\[
	\min_{x} \Vert Ax - b\Vert_2^2 \quad\text{ s.t. }\quad \Vert Lx\Vert_1\leq\tau.
\]
\subsection{Examples}
We will use two examples that are very different in nature in terms of their singular values. For both examples, we will show how their spectra are changed as a function of $\kappa$ by applying SR3, and how this relates to the inner and outer iterations. After that, we will show how our inexact SR3 greatly reduces the total number of iterations. We do not add noise to the data.
\subsubsection*{Gravity surveying}
The first example is the gravity example from the regu toolbox, \cite{Hansen2002,Hansen1994}. This example models gravity surveying. An unknown mass distribution that generates a gravity field is located in the subsurface, and the measured data is related to the gravity field via a Fredholm integral of the first kind, i.e.
\[
	b(s) = \displaystyle\int_{\Omega}k(s,t) x(t) dt.
\]
The variable $x(t)$ is the mass density at the location $t$ in the subsurface and $b(s)$ is the gravity field at location $s$ at the surface. The kernel is given by:
\[
	k(s,t) = d(d^2 + (s-t)^2)^{-3/2},
\]
where $d$ is the depth. The integral is discretized using the midpoint quadrature rule and yields a symmetric Toeplitz matrix $A$ that is square and severely ill-posed. We have chosen an $x(t)$ that is piecewise constant and hence we regularize the problem with TV regularization. The operator $L = D$, where $D$ is the first-order finite difference discretization, i.e.
\[
	D = \begin{bmatrix} -1 & 1 & & \\  &\ddots & \ddots & \\ & & -1 & 1\end{bmatrix} \in\mathbb{R}^{(n-1)\times n}.
\]
The operator is underdetermined and its nullspace has dimension 1. We choose $n = 512$. The true gravity profile is shown in \cref{fig:shepp-logan}.

\subsubsection*{Tomography}
Our second example is the tomography example PRtomo from the IR Tools toolbox \cite{Gazzola2020}, see also \cite{Hansen2018}, which models parallel tomography. It models X-ray attenuation tomography, often referred to as computerized tomography (CT). Parallel rays at different angles penetrate an object. The rays are attenuated at a rate proportional to the length of the ray and the density of the object. The $i$-th ray can be modeled as
\[
	b_i = \sum_{j\in\mathcal{S}_i} a_{ij}x_j.
\]
The set $\mathcal{S}$ denotes the set of pixels that are penetrated, $a_{ij}$ denotes the length of the $i$-th ray through the $j$-th pixel and $x_j$ is the attenuation coefficient. This is a 2D example where the matrix $A$ is underdetermined and the singular values decay mildly. Again, we use TV regularization for the reconstruction. For 2D regularization, the operator $L = \begin{bmatrix} I \otimes D \\ D \otimes I\end{bmatrix}$. Hence, the operator $L$ is overdetermined and has a nullspace of dimension 1. We choose 18 angles between 0 and 180 degrees and discretize the image on a $128\times 128$-pixel grid. This means that $A\in\mathbb{R}^{3258\times 16384}$. Our experiments are on the Shepp-Logan phantom, shown in \cref{fig:shepp-logan}. 
\begin{figure}[!h]
\centering
\begin{tabular}{cc}
	\includegraphics[width=0.4\textwidth]{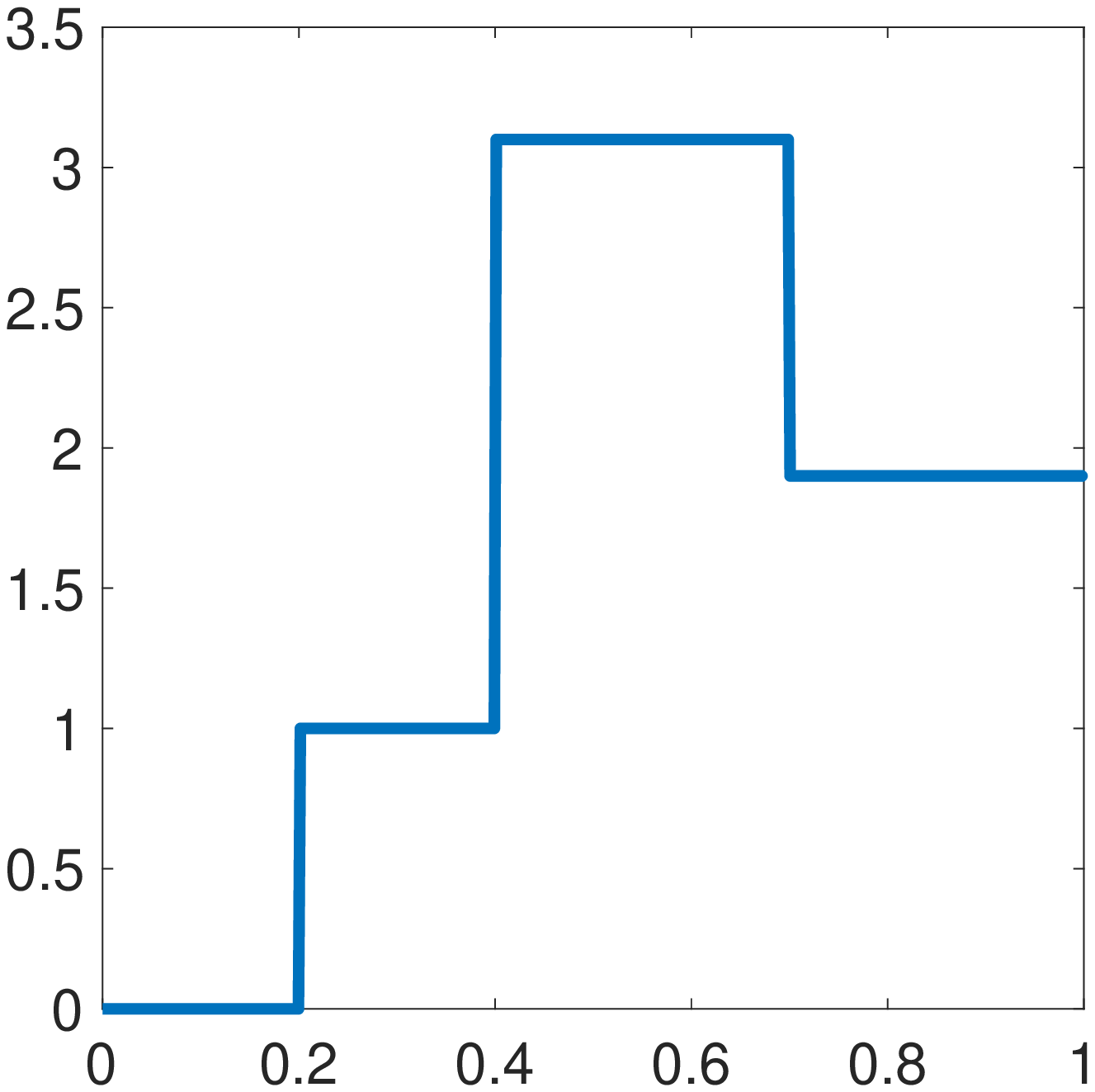} &
	\includegraphics[width=0.4\textwidth]{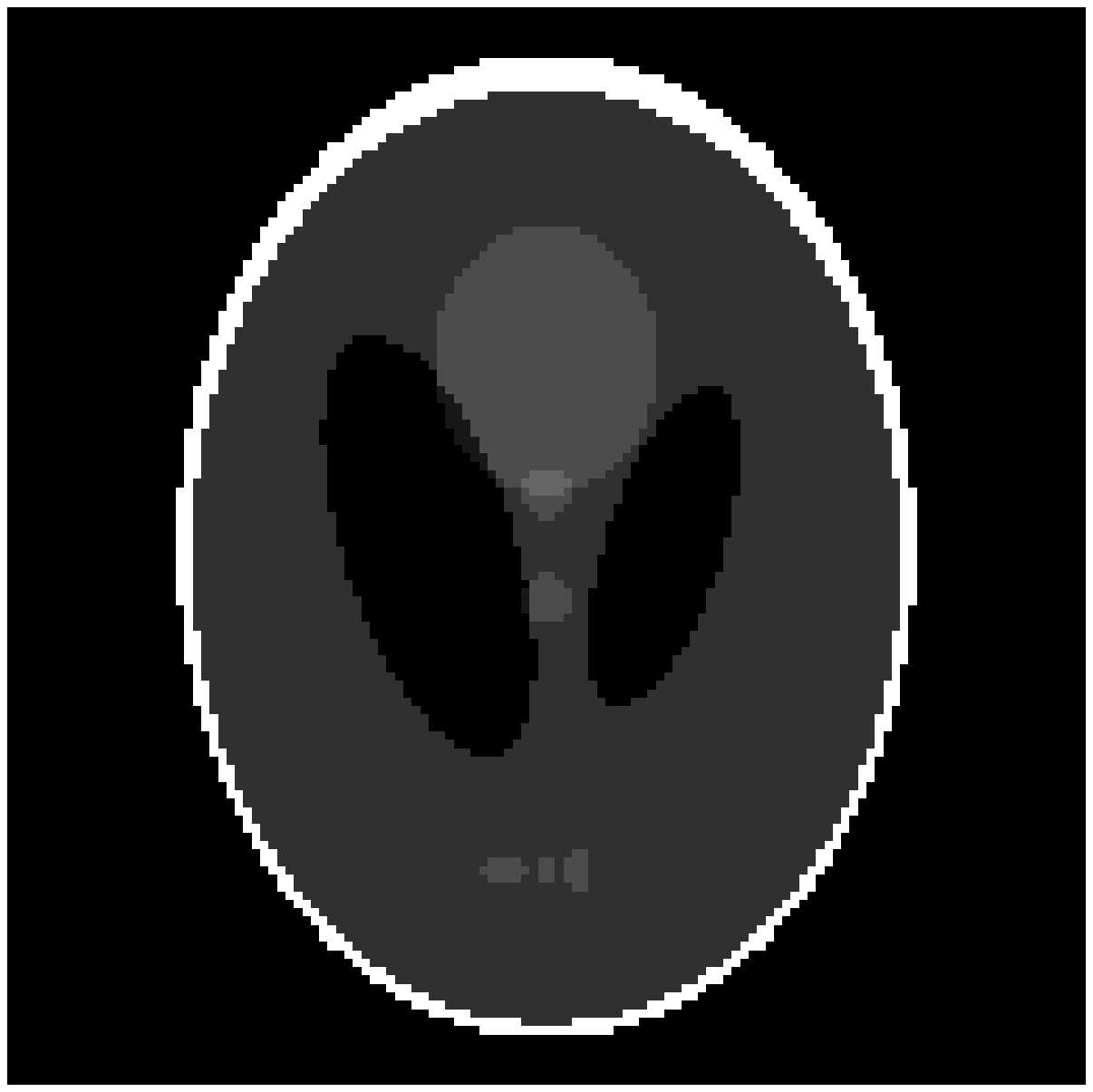} \\
	Gravity profile & Shepp-Logan phantom
\end{tabular}
\caption{}
\label{fig:shepp-logan}
\end{figure}

\subsubsection*{Parameters}
For our experiments, we have adapted the implementation of the accelerated proximal gradient algorithm from \cite{ODonoghue2016} for SR3 and use the same stopping criterion for the proximal gradient algorithm. For the inexact stopping criterion for the inner iteration we choose $\epsilon  = 10^{-6}$. For the exact SR3 method, we let LSQR run to convergence with the standard tolerance of $10^{-6}$. For $\tau$, we choose the optimal value $\tau = \Vert Lx_{\text{true}}\Vert_1$.

\subsection{Singular values of $F_{\kappa}$}
In this section we show the singular values of $F_{\kappa}$ for the gravity and the tomography example. For the tomography example, the generalized singular values are calculated on a $64\times 64$ grid to reduce computational time, instead of the $128\times 128$ grid for our experiments. We show the generalized singular values , i.e. the singular values of $AL_{A}^{\dagger}$, and the singular values of $F_{\kappa}$ for different values of $\kappa$ for the gravity example in \cref{fig:gravity-svd}. 
\begin{figure}[!h]
	\centering
	\begin{tabular}{cc}
		\includegraphics[width=0.4\linewidth]{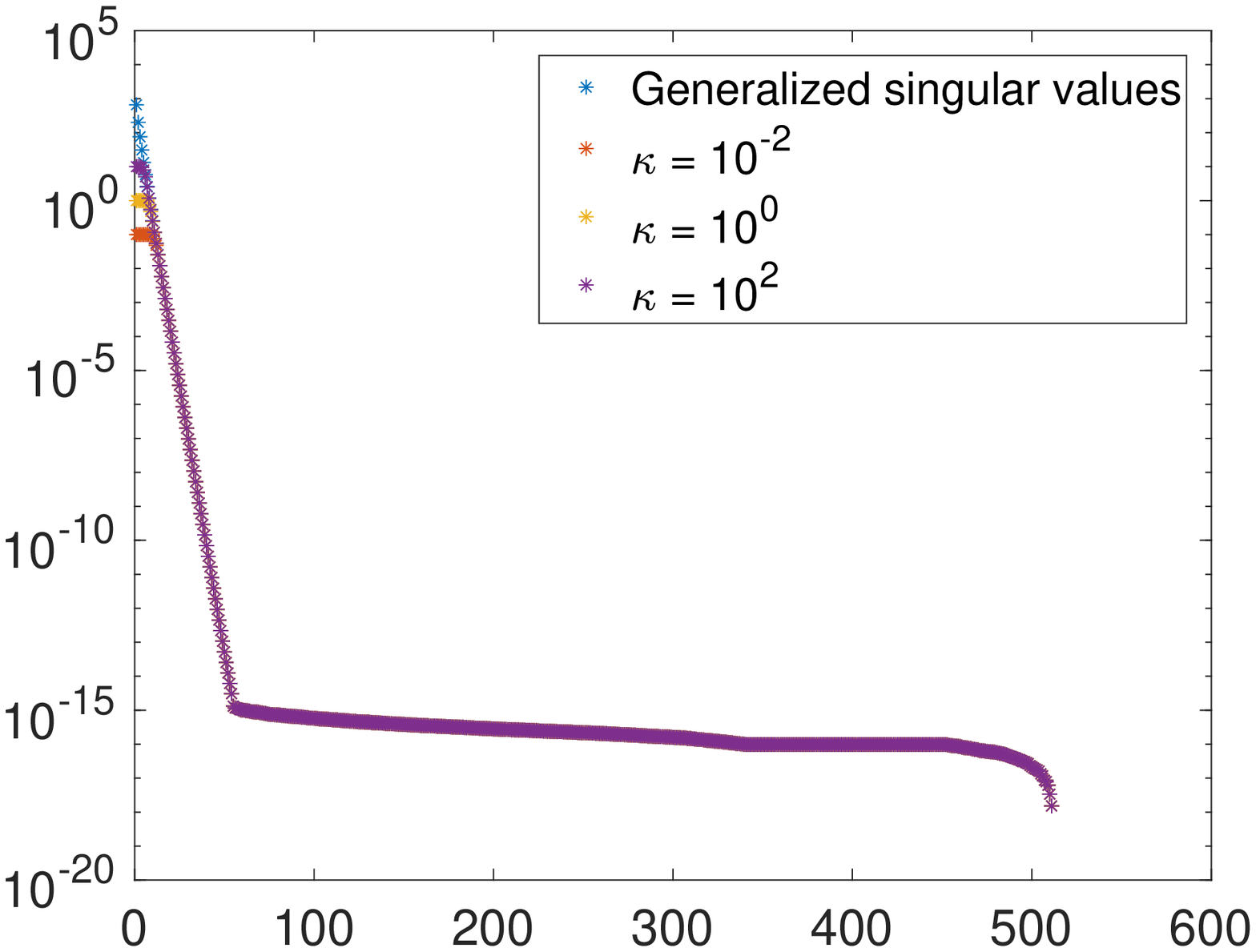} &
		\includegraphics[width=0.4\linewidth]{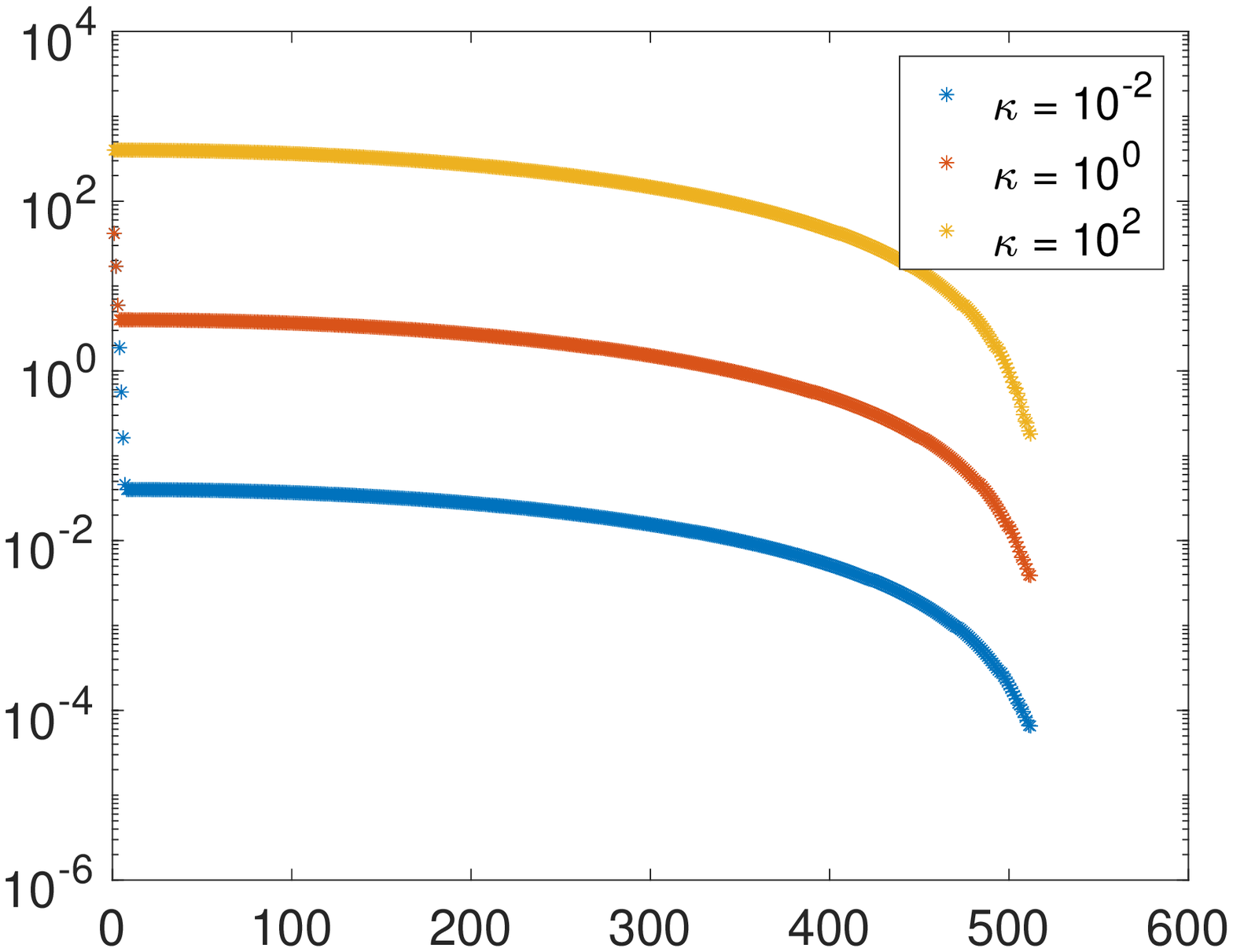} \\
		Singular values of $F_{\kappa}$ and $AL_{A}^{\dagger}$. & Singular values of $H_{\kappa}$.
	\end{tabular}
\caption{Spectral properties of $F_{\kappa}$ and $H_{\kappa}$ for the gravity example. Left figure: We show the singular values of $AL_{A}^{\dagger}$ and the singular values of $F_{\kappa}$ for different values of $\kappa$. Note that the singular values of $F_{\kappa}$ have a very similar structure to the singular values of $AL_{A}^{\dagger}$. Right figure: The singular values of the matrix $H_{\kappa}$.}
\label{fig:gravity-svd}
\end{figure}
Note that irrespective of the value of $\kappa$, the matrix $F_{\kappa}$ remains severely ill-posed. For the tomography exmaple, $A$ is not severely ill-posed. The singular values decay only mildly and the situation is different. In this case, for small $\kappa$,
\[
	\psi_i\left(F_{\kappa}\right) = \sqrt{\dfrac{\sigma_{r-i+1}^2}{\sigma_{r-i+1}^2/\kappa+ \gamma_{r-i+1}^2}} \approx \sqrt{\dfrac{\sigma_{r-i+1}^2}{\sigma_{r-i+1}^2/\kappa}} = \sqrt{\kappa}.
\]
Hence, for small $\kappa$ the singular values of $F_{\kappa} \approx\sqrt{\kappa}$ and the condition number is 1. As $\kappa\to\infty$ we have seen that $\psi_i(F_{\kappa})\to \dfrac{\sigma_{r-i+1}}{\gamma_{r-i+1}}$. We show the singular values, the generalized singular values, and the singular values of $F_{\kappa}$ in \cref{fig:tomo-svd}. Note that for this example, the conditioning of the matrix $F_{\kappa}$ is improved.
\begin{figure}[!h]
	\centering
	\begin{minipage}{\textwidth}
		\begin{tabular}{cc}
			\includegraphics[width=0.4\linewidth]{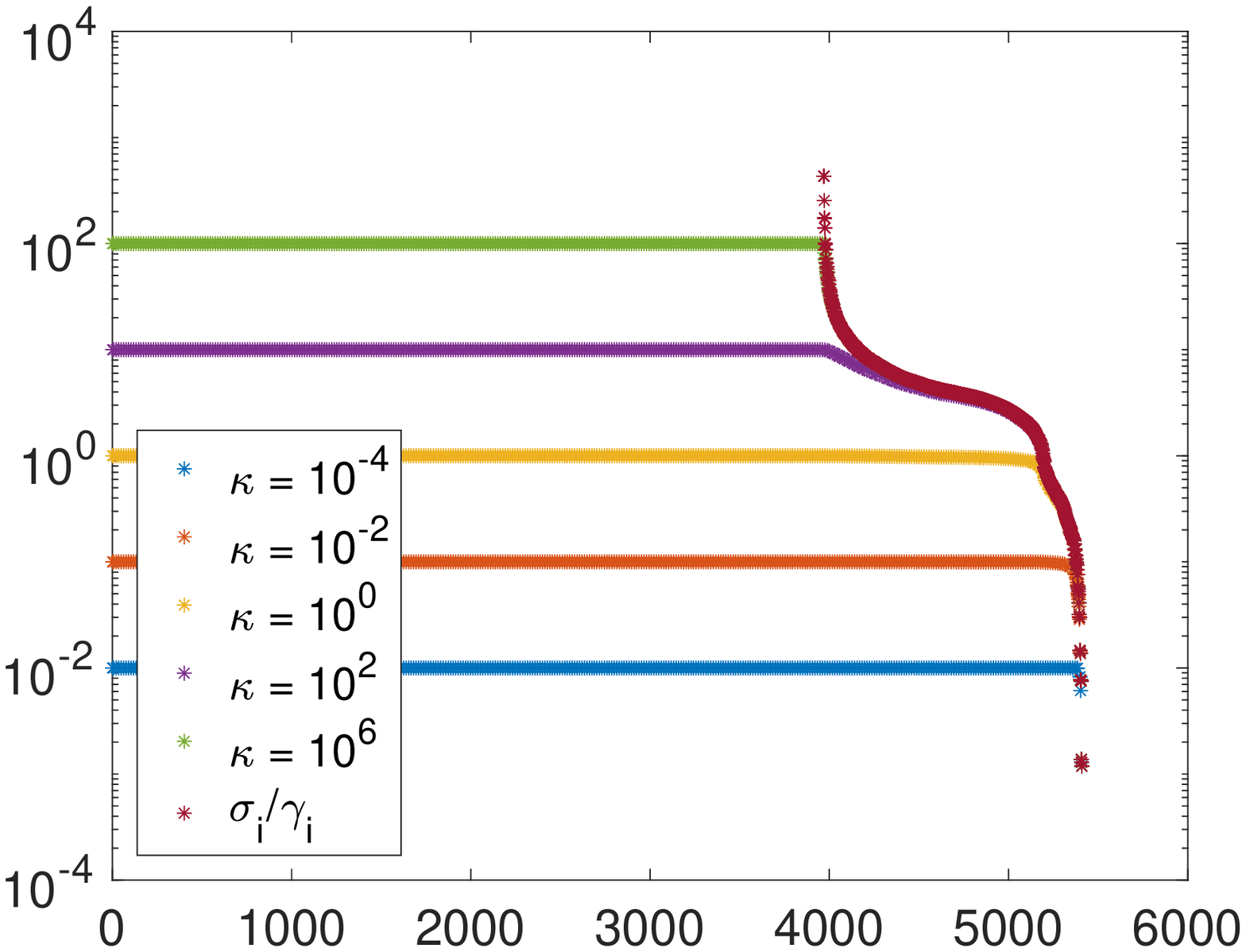} &
			\includegraphics[width=0.4\linewidth]{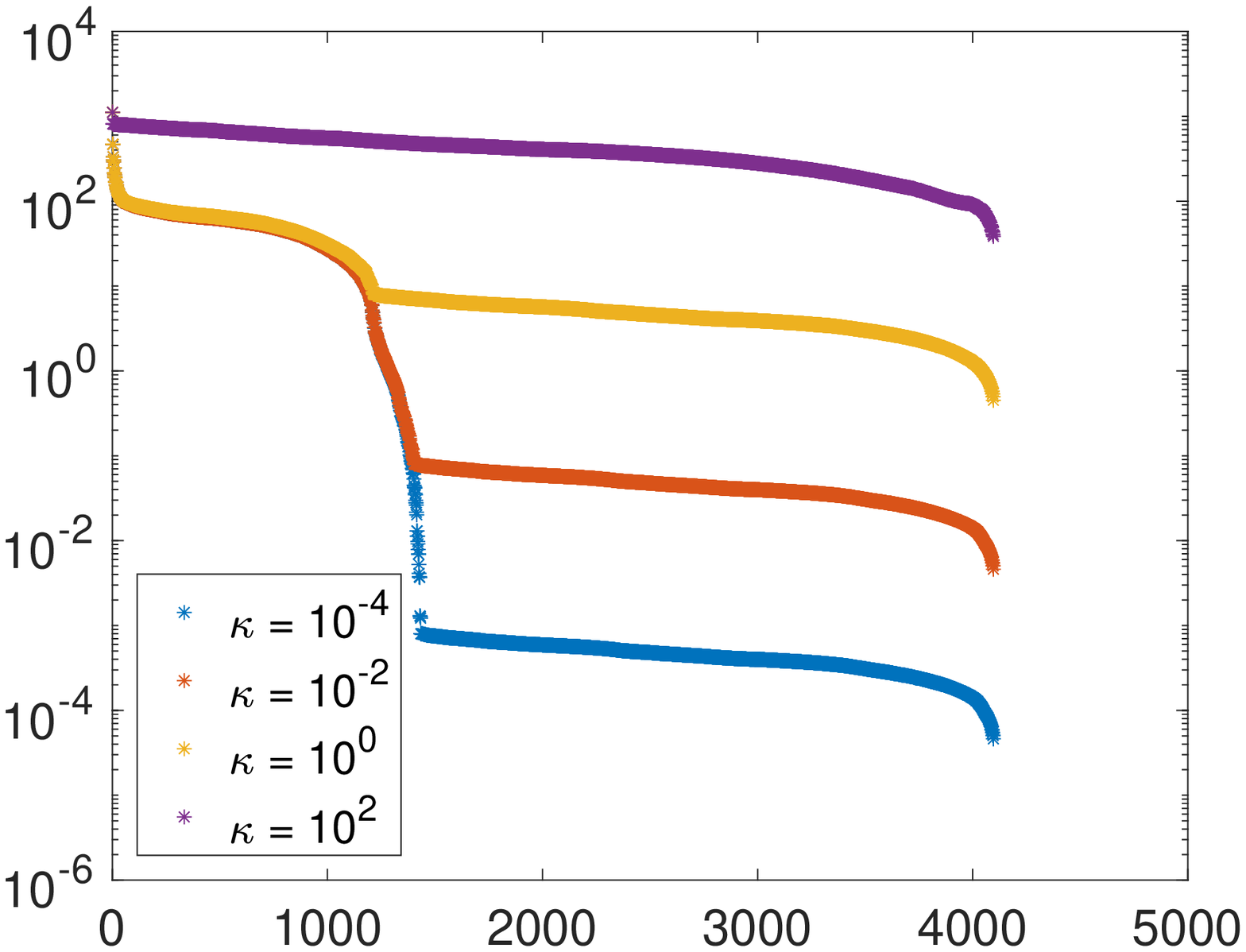} \\
			Singular values of $F_{\kappa}$.\footnote{\label{footnote1}The matrix is numerically rank deficient and we have truncated the SVD.} & Singular values of $H_{\kappa}$. \\
		\end{tabular}
	\end{minipage}
\caption{Spectral properties of $F_{\kappa}$ and $H_{\kappa}$ for the tomography example. The left figure shows the singular values of $F_{\kappa}$. Recall that the first $p-r_{L}$ singular values of $F_{\kappa}$ are $\sqrt{\kappa}$. The right figure shows the singular values of $H_{\kappa}$. There is an inverse relation between the condition number of $H_{\kappa}$ and $F_{\kappa}$ as a function of $\kappa$.}
\label{fig:tomo-svd}
\end{figure}

\subsection{The Pareto curves}
In figure \cref{fig:pareto-curve} we show the Pareto curves for the original problem and SR3 for both our examples.
\begin{figure}[!h]
	\centering
	\begin{tabular}{cc}
	\includegraphics[width=0.4\linewidth]{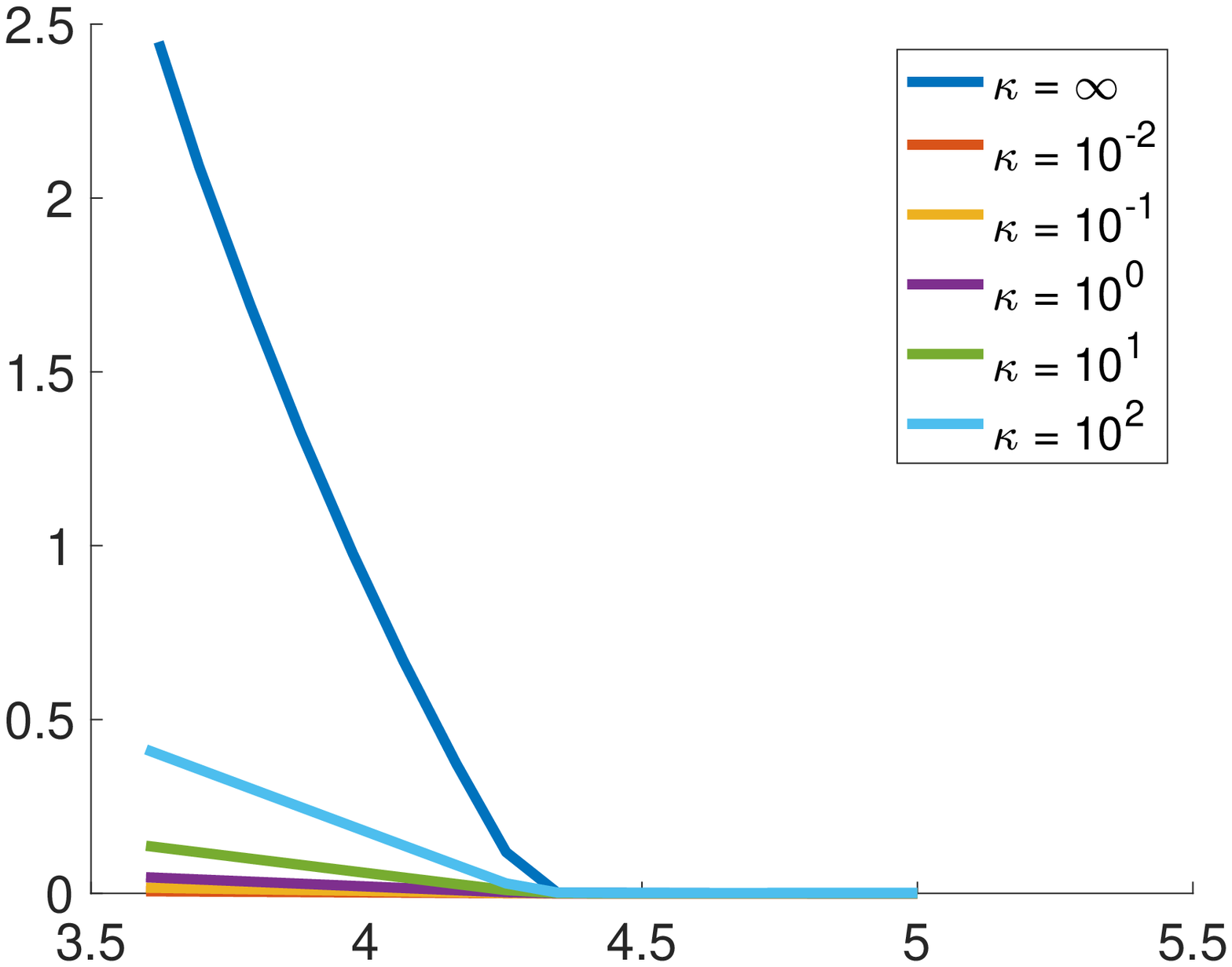} &
	\includegraphics[width=0.4\linewidth]{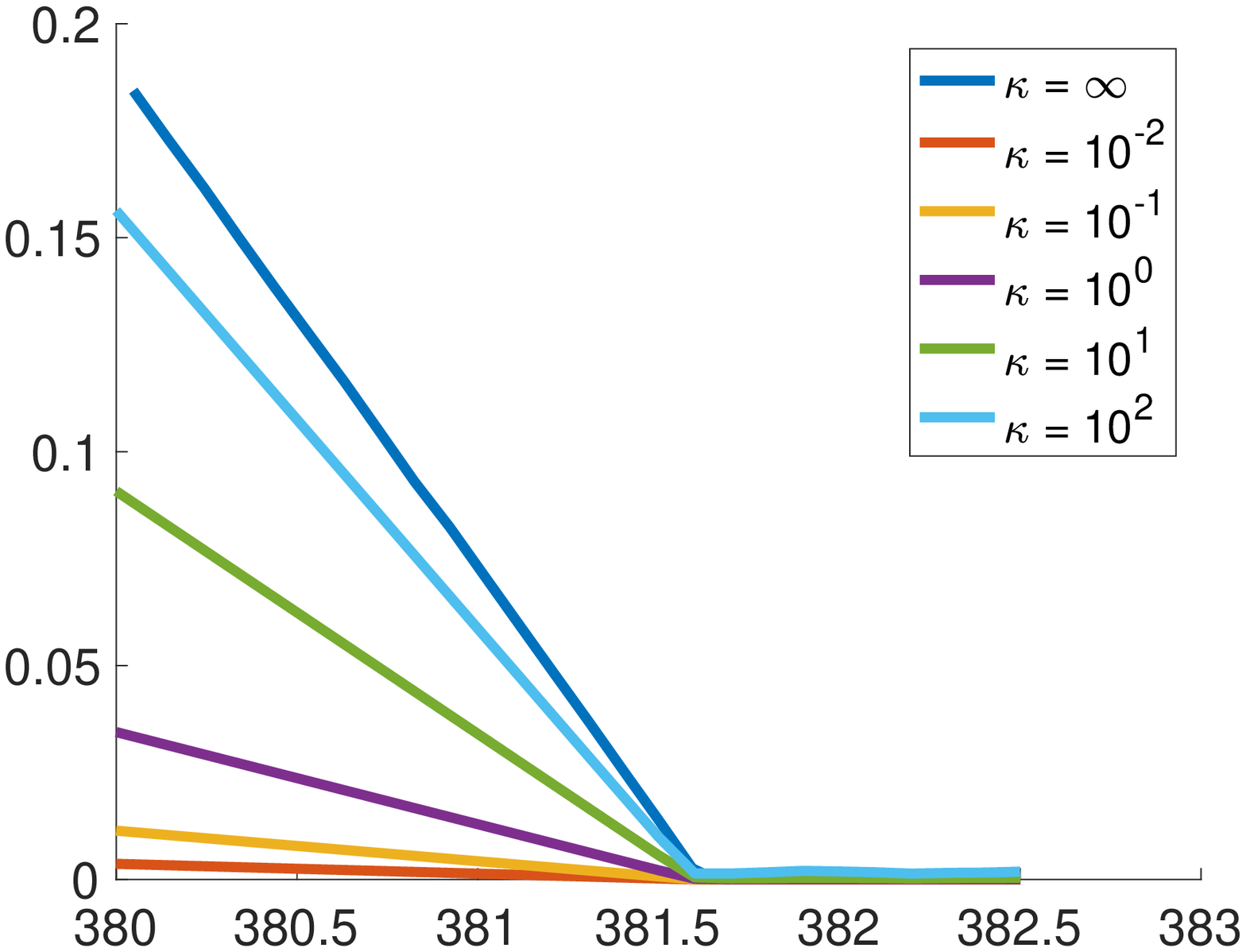} \\
	Pareto curves for the gravity example. & Pareto curves for the tomography example.
	\end{tabular}
	\caption{The left figure shows the Pareto curves for the gravity example. The right figure shows the Pareto curves for the tomography example. The x-axis is $\tau$ and the y-axis is $\Vert A\bar{x}_{\kappa} - b\Vert_2$.}
	\label{fig:pareto-curve}
\end{figure}
As we explained in section \cref{sec:value_function}, the corner of the Pareto of the original problem and SR3 is likely to be in the same place. This is confirmed by figure \cref{fig:pareto-curve}.

\subsection{The influence of $\kappa$ on the number of iterations}
To investigate the influence of $\kappa$, we show the amount of inner and outer iterations for varying values of $\kappa$ and the total number of iterations. The results are shown in \cref{fig:iterations-gravity} and \cref{fig:iterations-tomo}. As we have stated before, the improved convergence rate due to an improved conditioning of $\kappa$ pertains to the outer iterations. The effect of $\kappa$ on the convergence of the inner iteration may be completely opposite. \\
\begin{figure}[!h]
	\centering
	\begin{tabular}{cc}
	\includegraphics[width=0.4\linewidth]{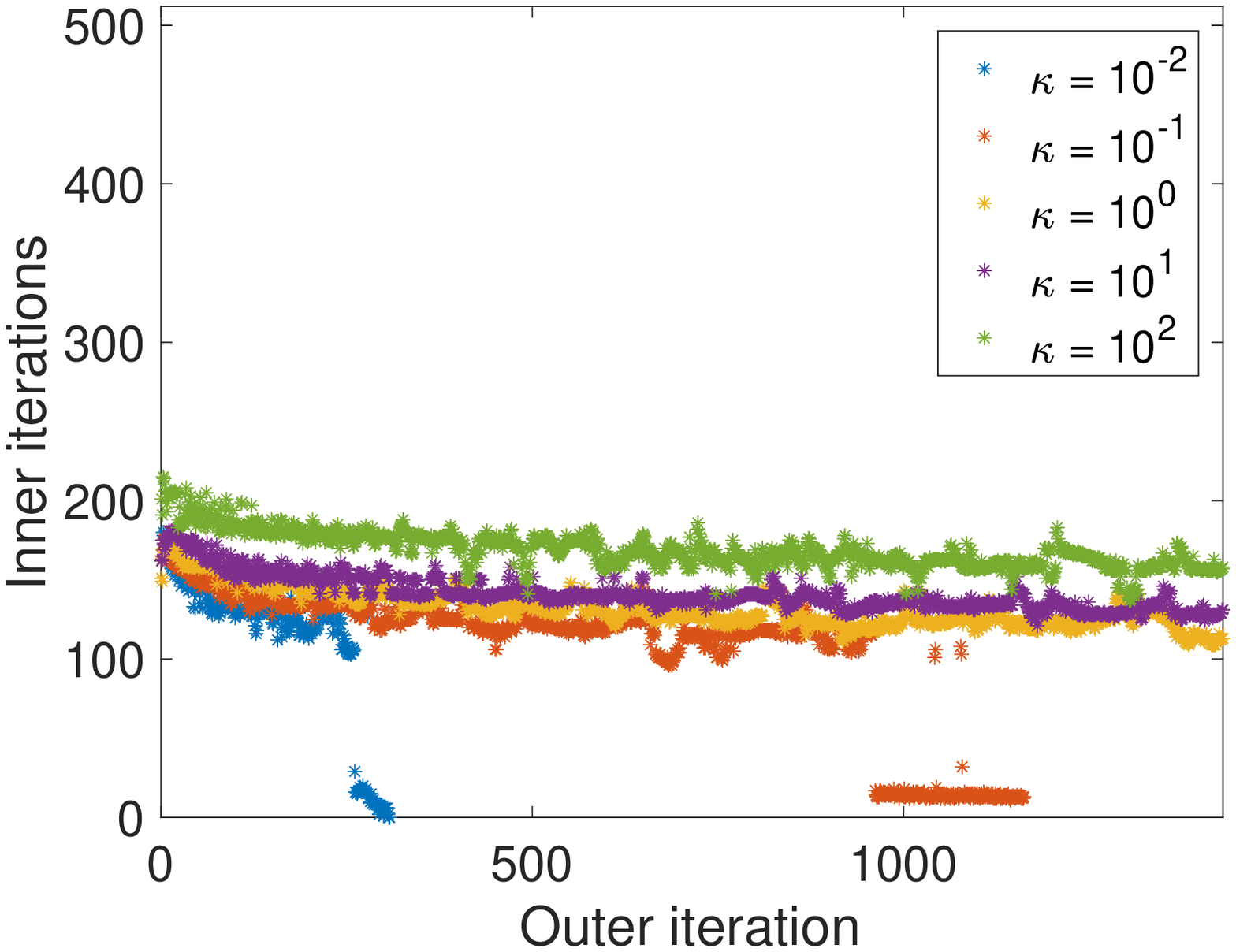} &
	\includegraphics[width=0.4\linewidth]{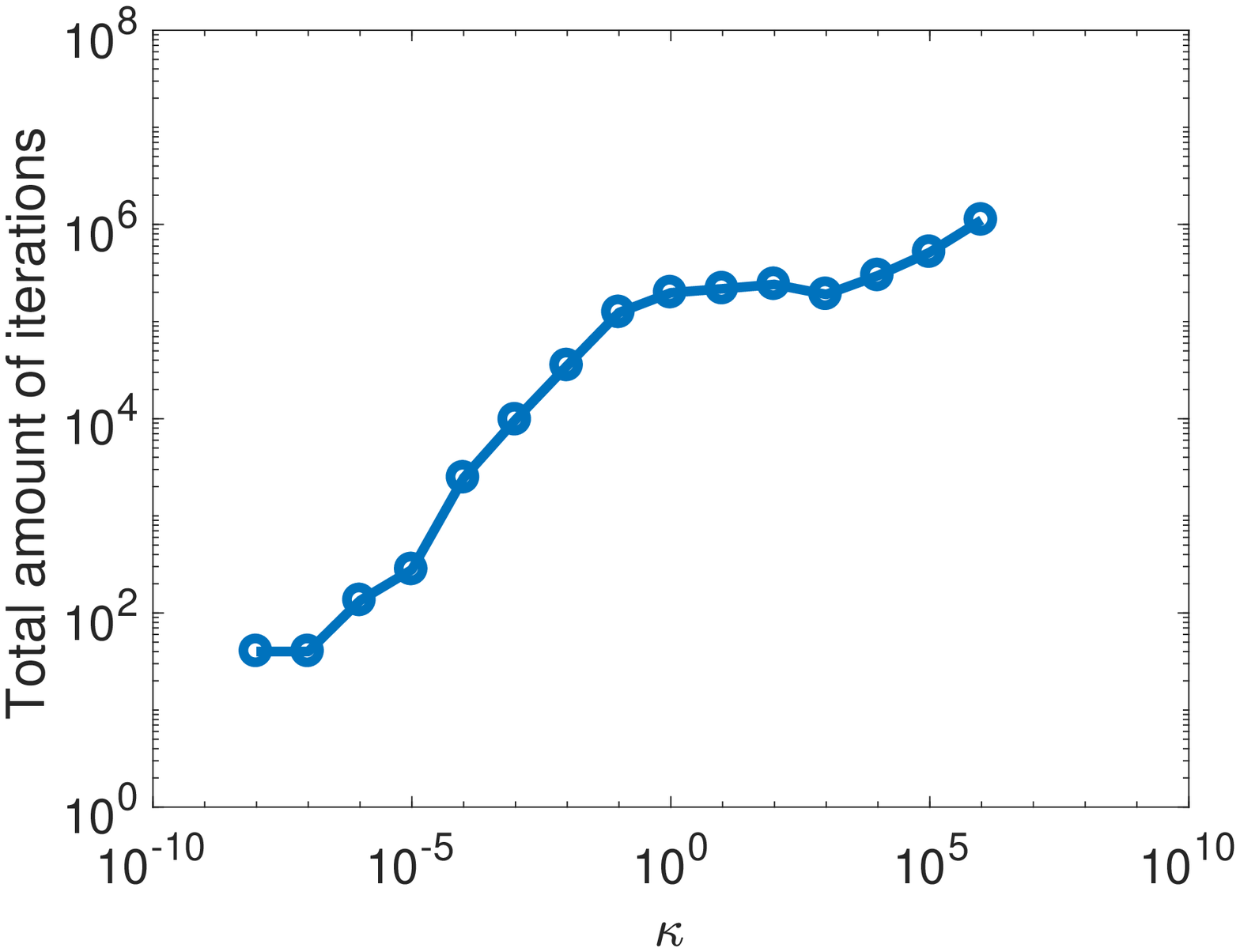} \\
	Inner iterations versus outer iterations. & Total number of iterations.
	\end{tabular}
	\caption{The left figure shows the inner and outer iterations for varying $\kappa$ for the gravity example. The right figure shows the total number of iterations.}
	\label{fig:iterations-gravity}
\end{figure}
\begin{figure}[!h]
	\centering
	\begin{tabular}{cc}
	\includegraphics[width=0.4\linewidth]{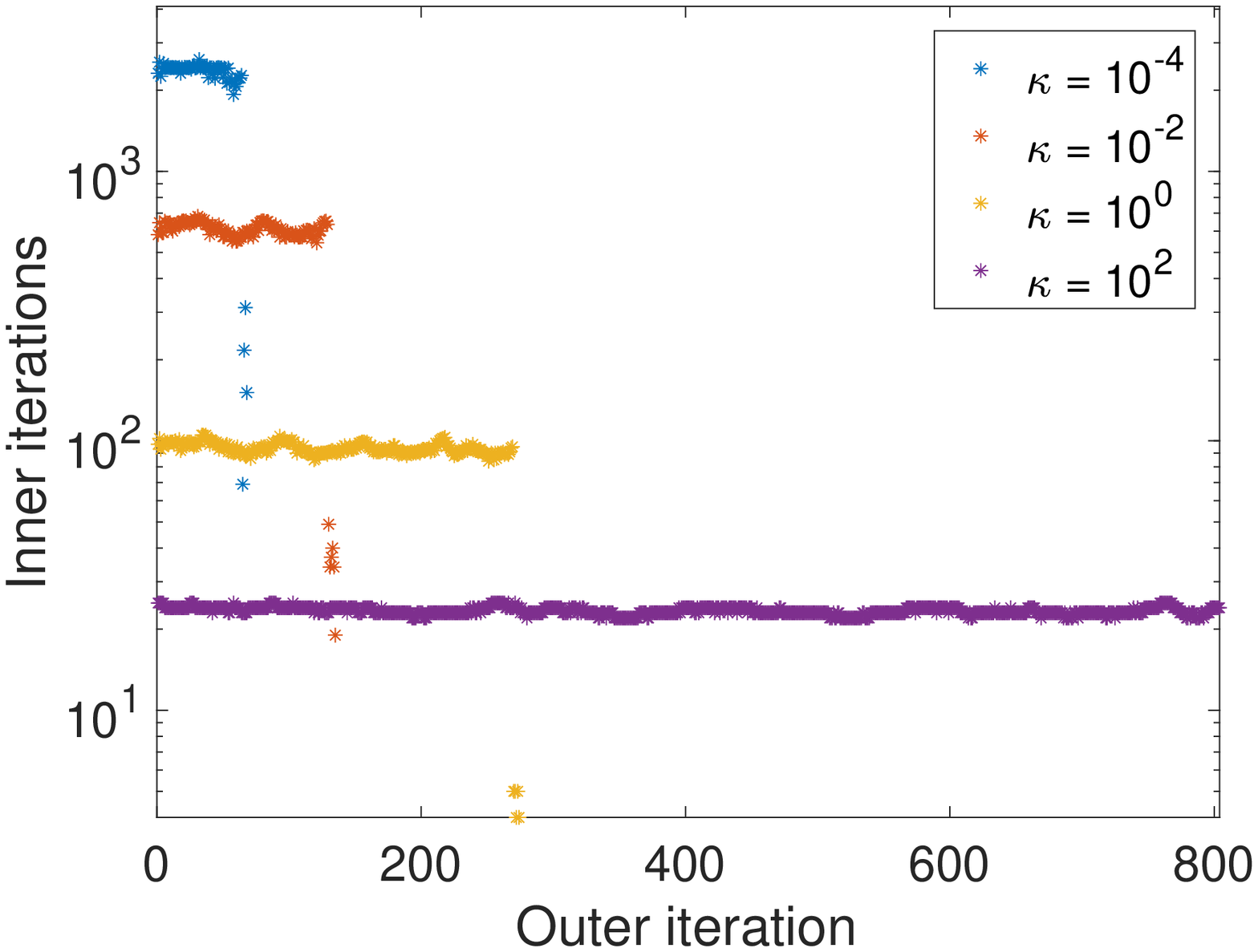} &
	\includegraphics[width=0.4\linewidth]{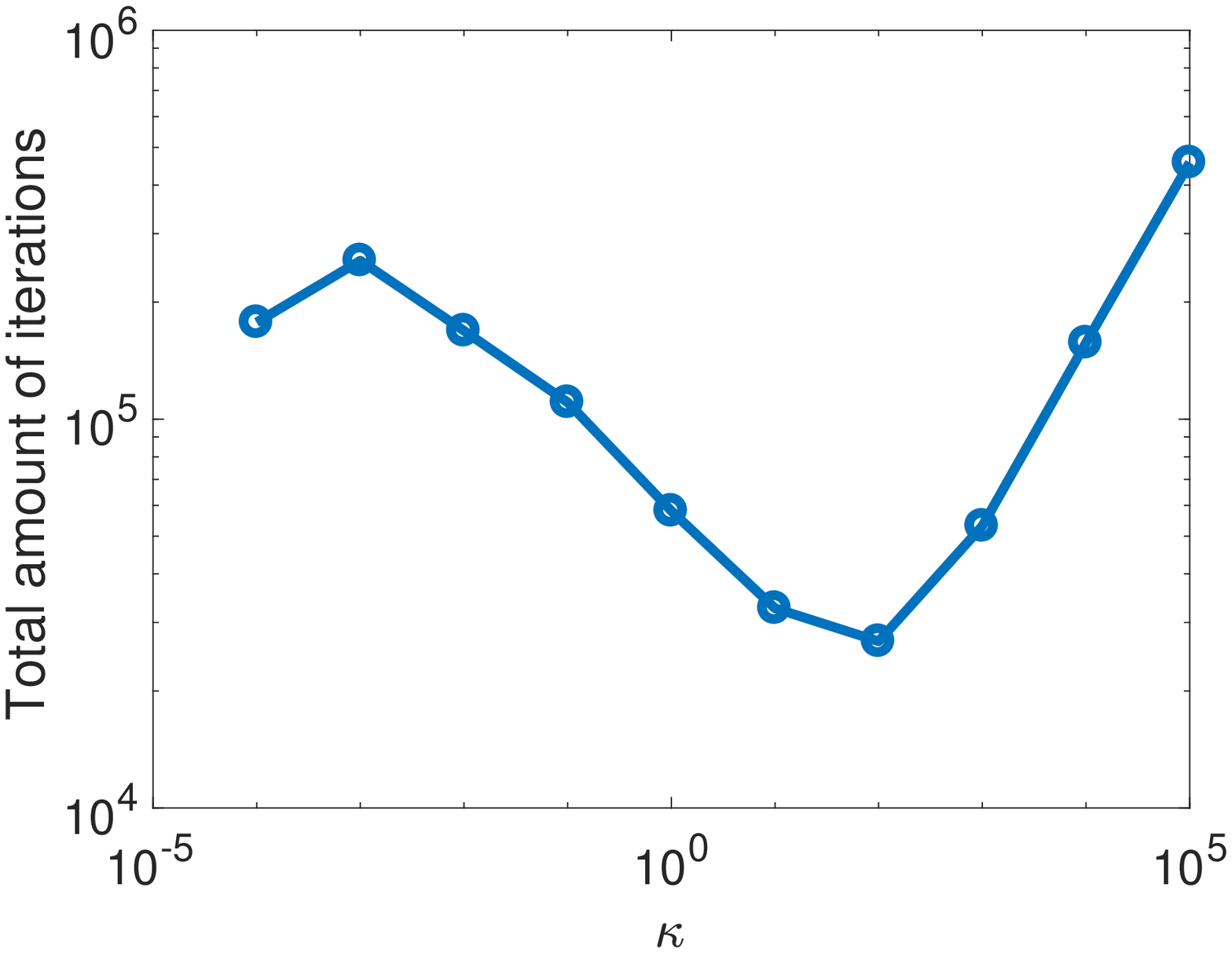} \\
	Inner iterations versus outer iterations. & Total number of iterations.
	\end{tabular}
	\caption{The left figure shows the inner and outer iterations for varying $\kappa$ for the tomography example. The right figure shows the total number of iterations.}
	\label{fig:iterations-tomo}
\end{figure}
For the gravity example, we see that the amount of inner iterations varies very little as $\kappa$ increases, and even goes up a little bit. This is not unexpected, because the decay of the singular values changes very little as $\kappa$ increases, see \cref{fig:gravity-svd}. The amount of outer iterations goes down rapidly as $\kappa$ decreases, something that is not expected from the distribution of the singular values. This shows that the distribution of the singular values is not the sole property explaining the convergence behavior. \\
For the tomography example we see a clear trade-off between inner and outer iterations. From \cref{fig:tomo-svd} we clearly see that as the condition number of $F_{\kappa}$ decreases, the condition number of $H_{\kappa}$ increases. This explains that, as the amount of inner iterations goes down with increasing $\kappa$, the amount of outer iterations goes down.

\subsection{Inexact SR3}
In this section we compare the error and the total number of iterations for SR3 and inexact SR3 as a function of $\kappa$. The results are shown in \cref{fig:inexact_SR3}.
\begin{figure}[!h]
	\centering
	\begin{tabular}{cc}
		\includegraphics[width=0.4\linewidth]{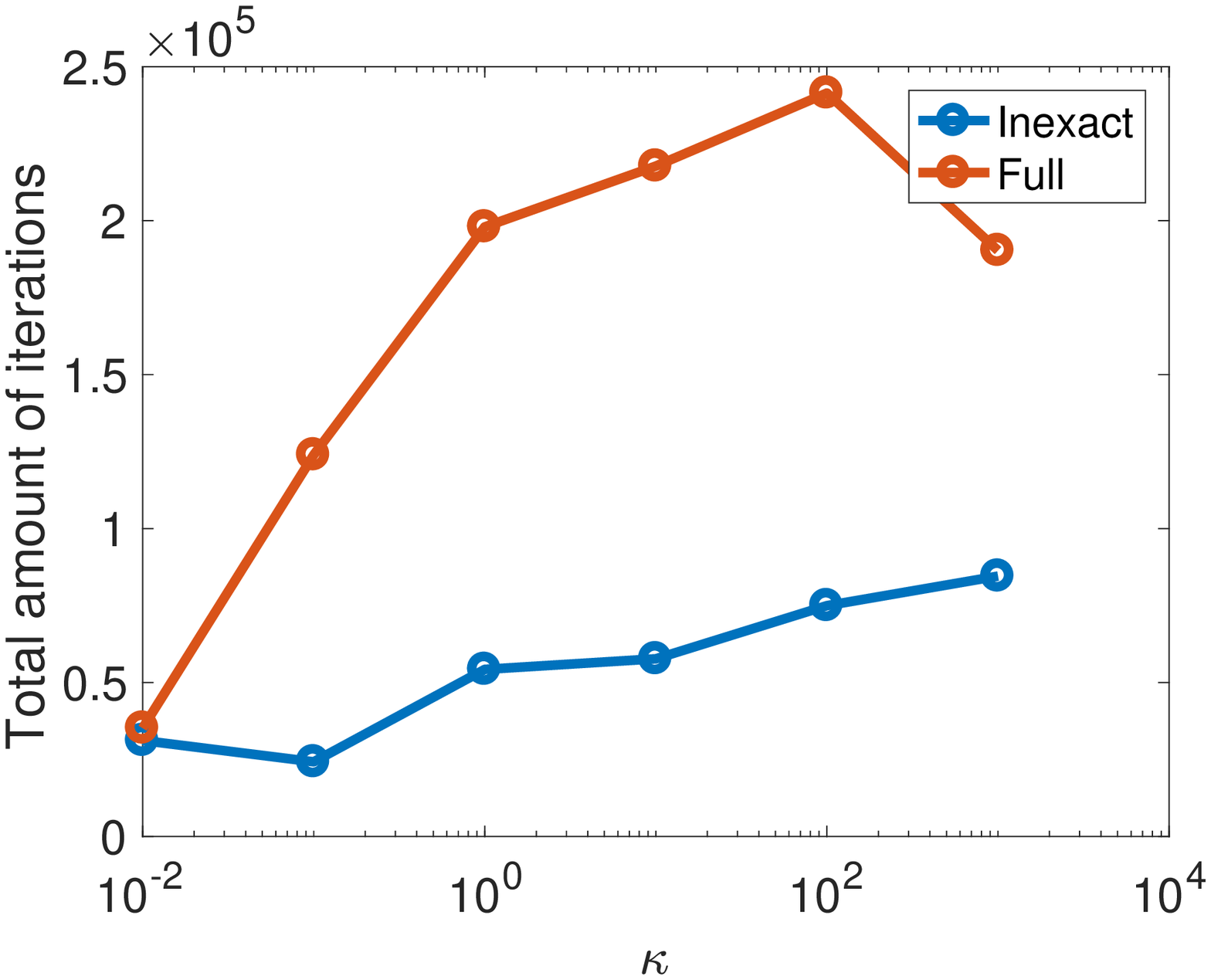} &	 \includegraphics[width=0.4\linewidth]{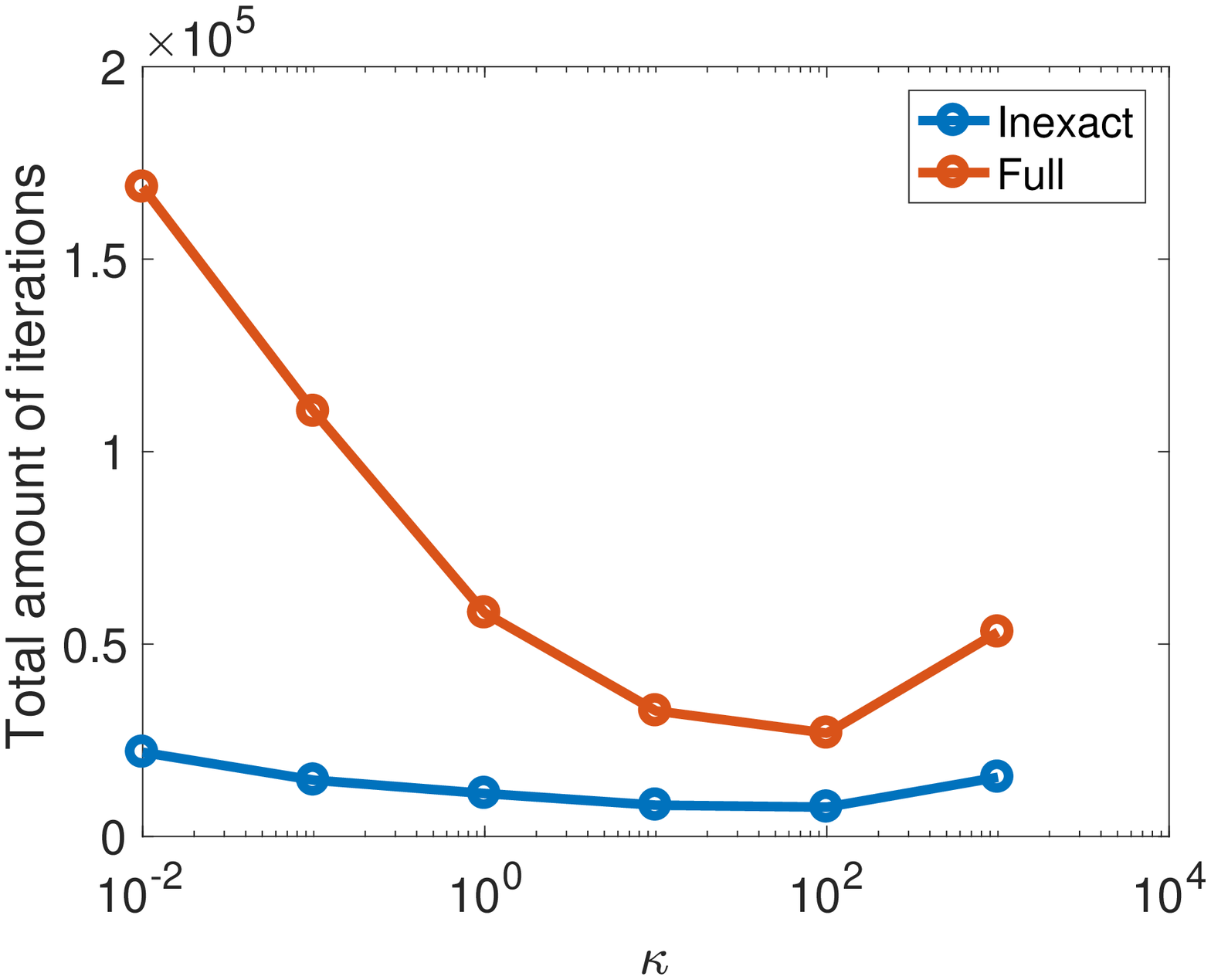} \\
		Gravity example & Tomography example \\
	\end{tabular}
	\caption{Comparison of the total number of iterations for SR3 and inexact SR3 as a function of $\kappa$. Note that the axes are on a log-log scale.}
	\label{fig:inexact_SR3}
\end{figure}
We see that the total number of iterations needed is greatly reduced by implementing the automated stopping criterion. Another important contribution is that the stopping criterion seems to mitigate the influence of $\kappa$ on the total amount of iterations. Figures \cref{fig:gravity-kappa} and \cref{fig:tomo-kappa} show some reconstructions for different values of $\kappa$.
\begin{figure}[!h]
\centering
		\begin{tabular}{cc}
			\includegraphics[width=0.5\linewidth]{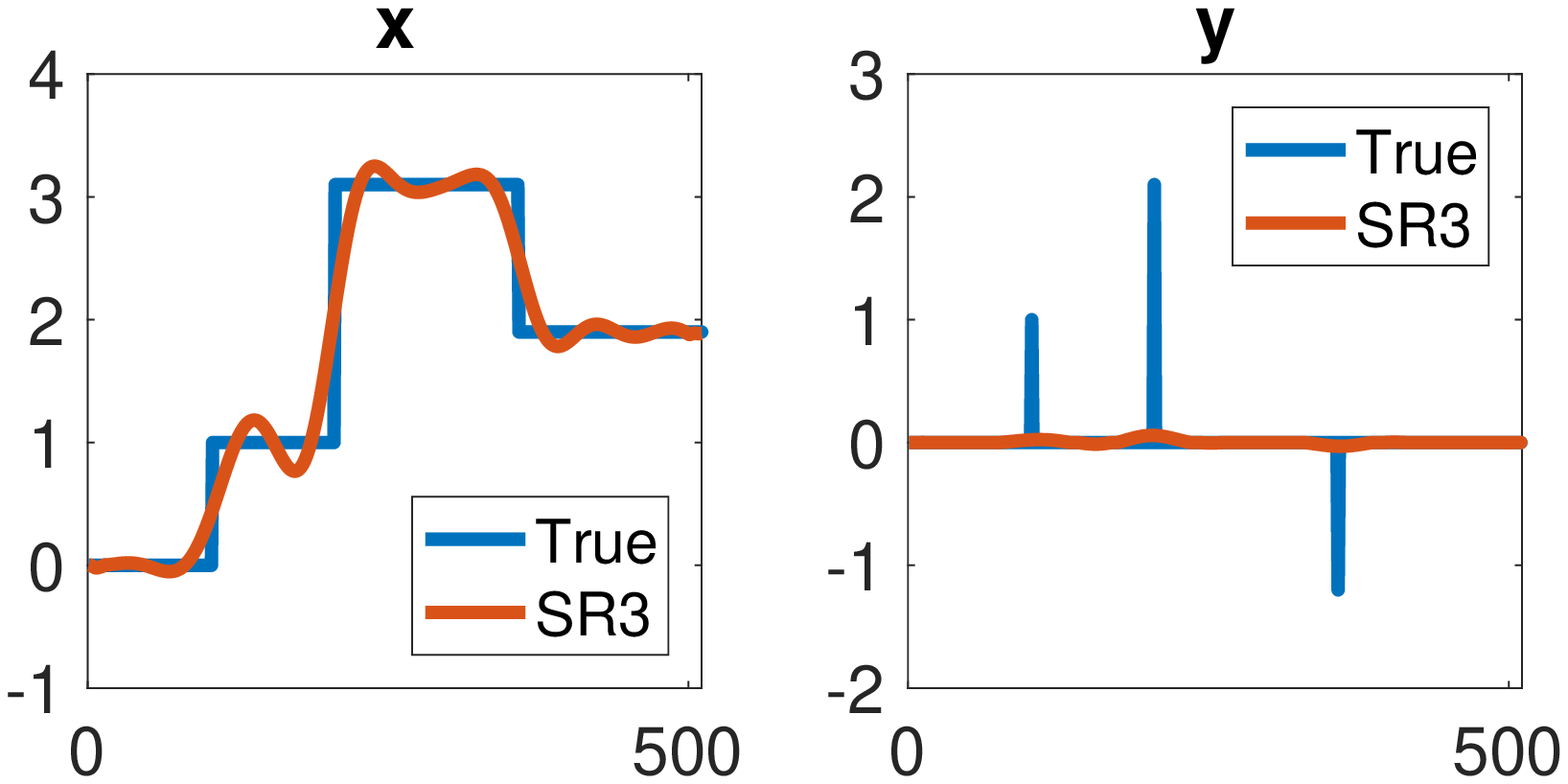} &
			\includegraphics[width=0.5\linewidth]{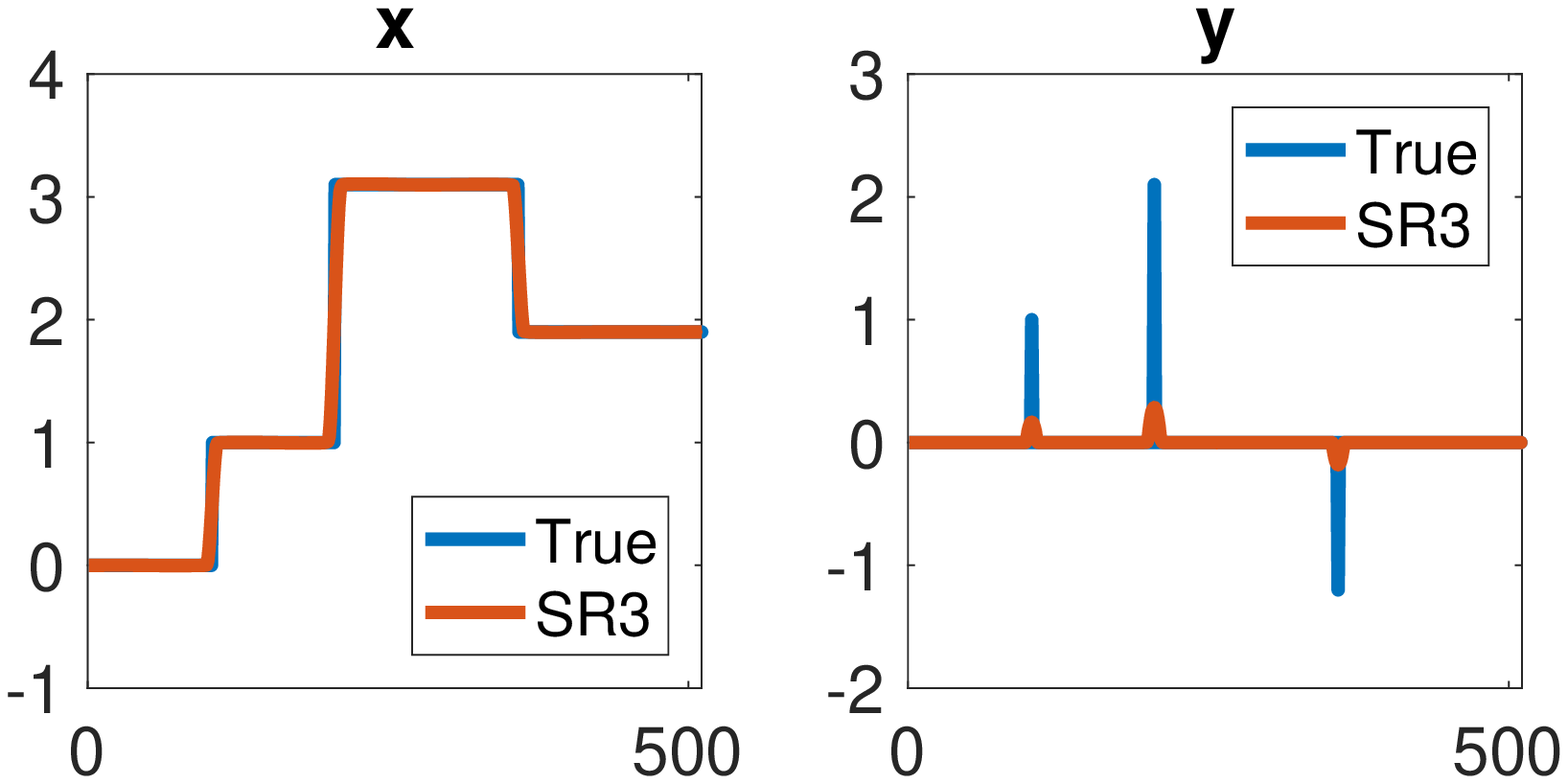} \\
			$\kappa = 10^{-6}$. & $\kappa = 10^{-2}$. \\
			\includegraphics[width=0.5\linewidth]{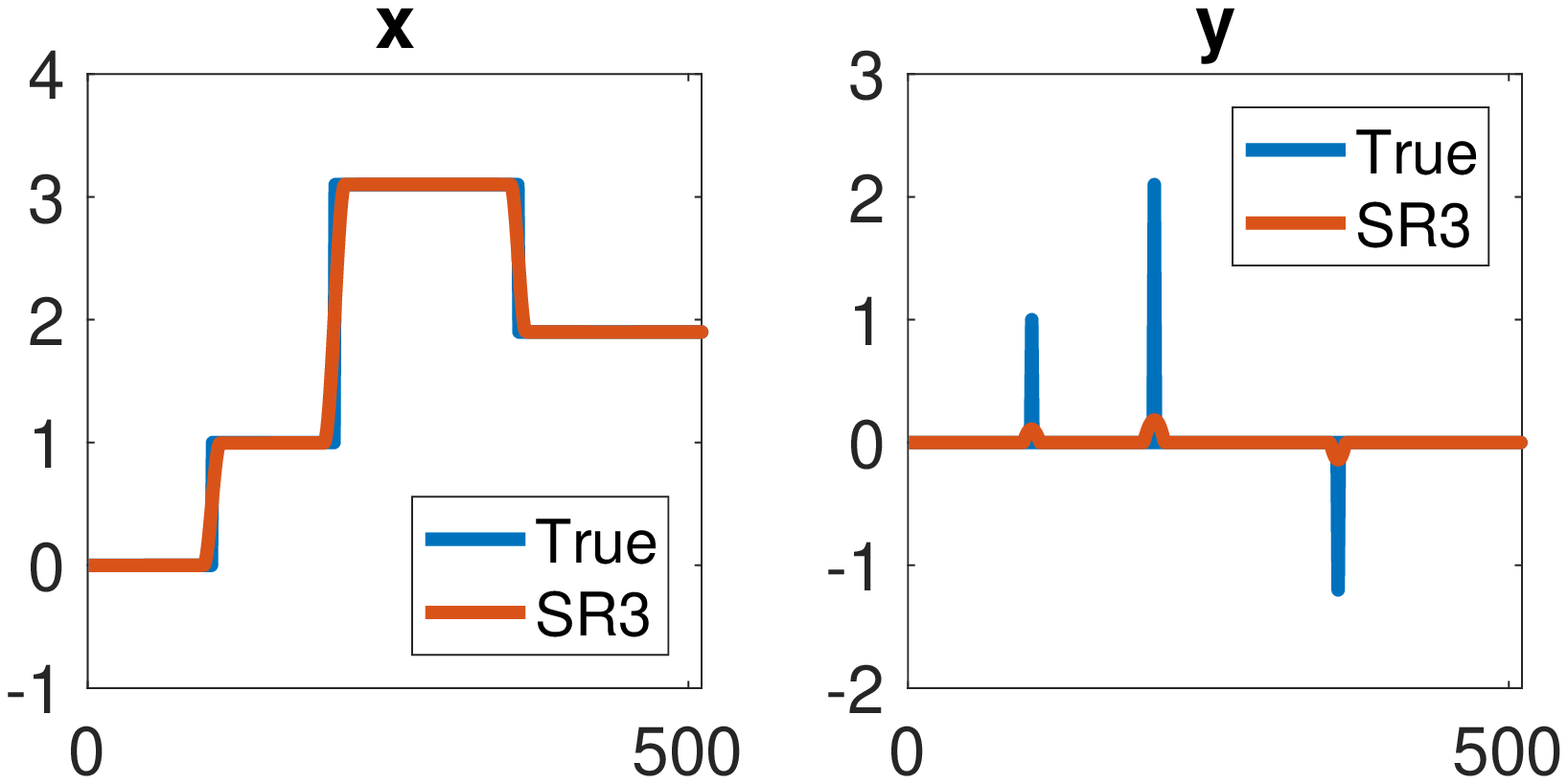} &
			\includegraphics[width=0.5\linewidth]{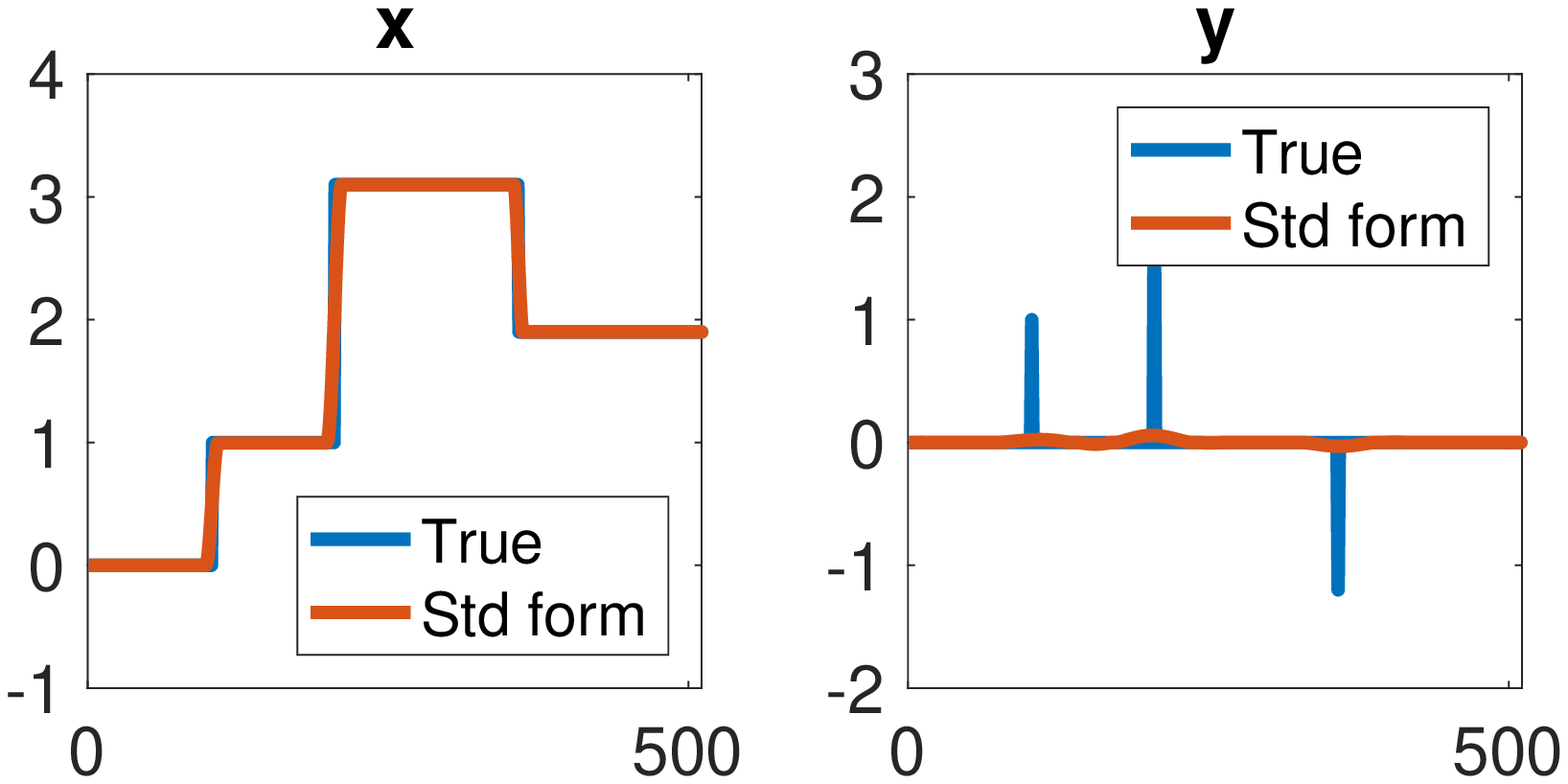} \\
			$\kappa = 10^{2}$. & Standard-form transformation. \\
		\end{tabular}
\caption{Solution to the gravity example for different $\kappa$. We have chosen $\lambda$ by hand to yield the best reconstruction. We show both $\bar{x}_{\kappa}$ and $\bar{y}_{\kappa}$.}
\label{fig:gravity-kappa}
\end{figure}
\begin{figure}[!h]
\centering
		\begin{tabular}{cc}
			\includegraphics[width=0.4\linewidth]{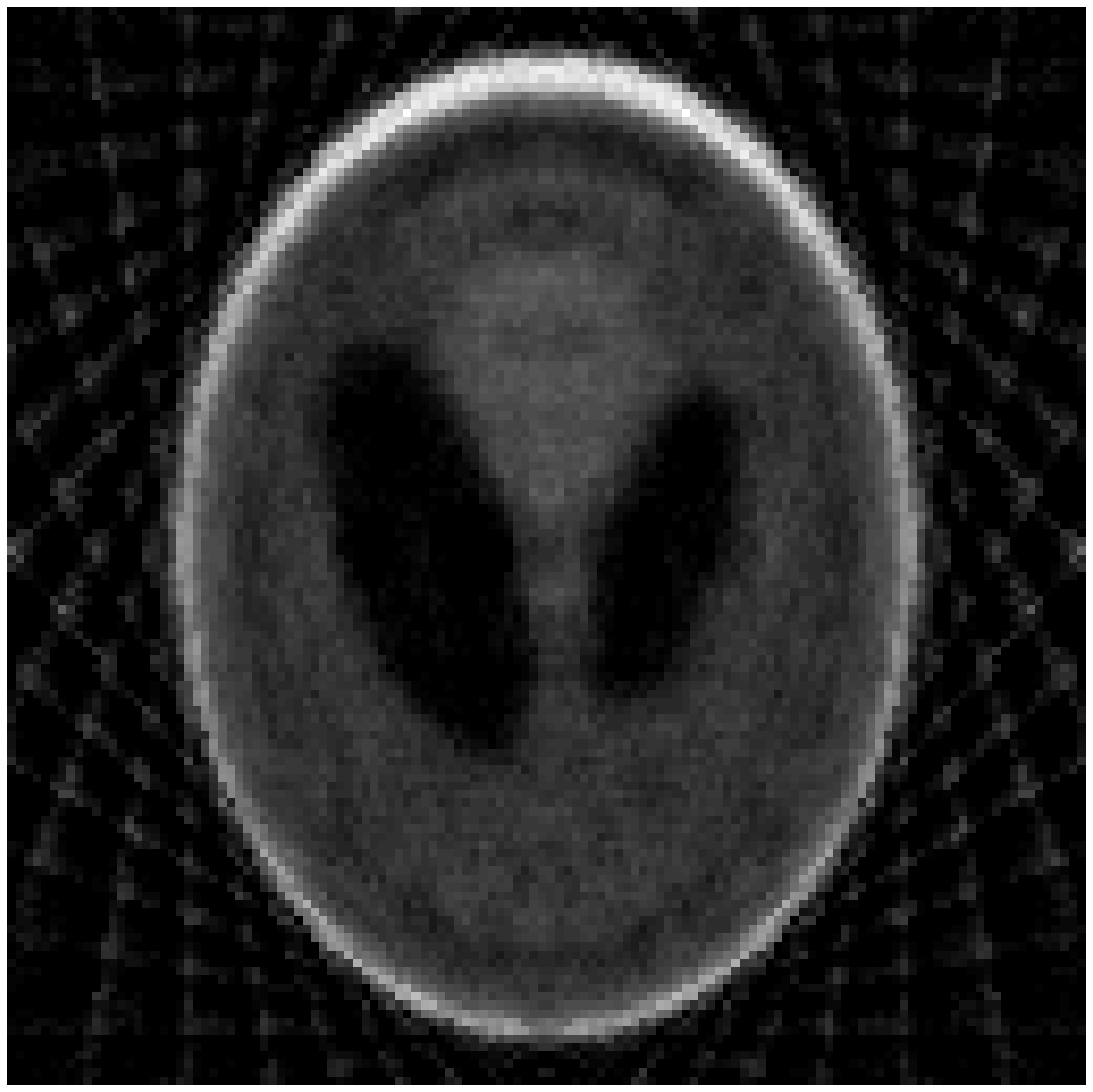} &
			\includegraphics[width=0.4\linewidth]{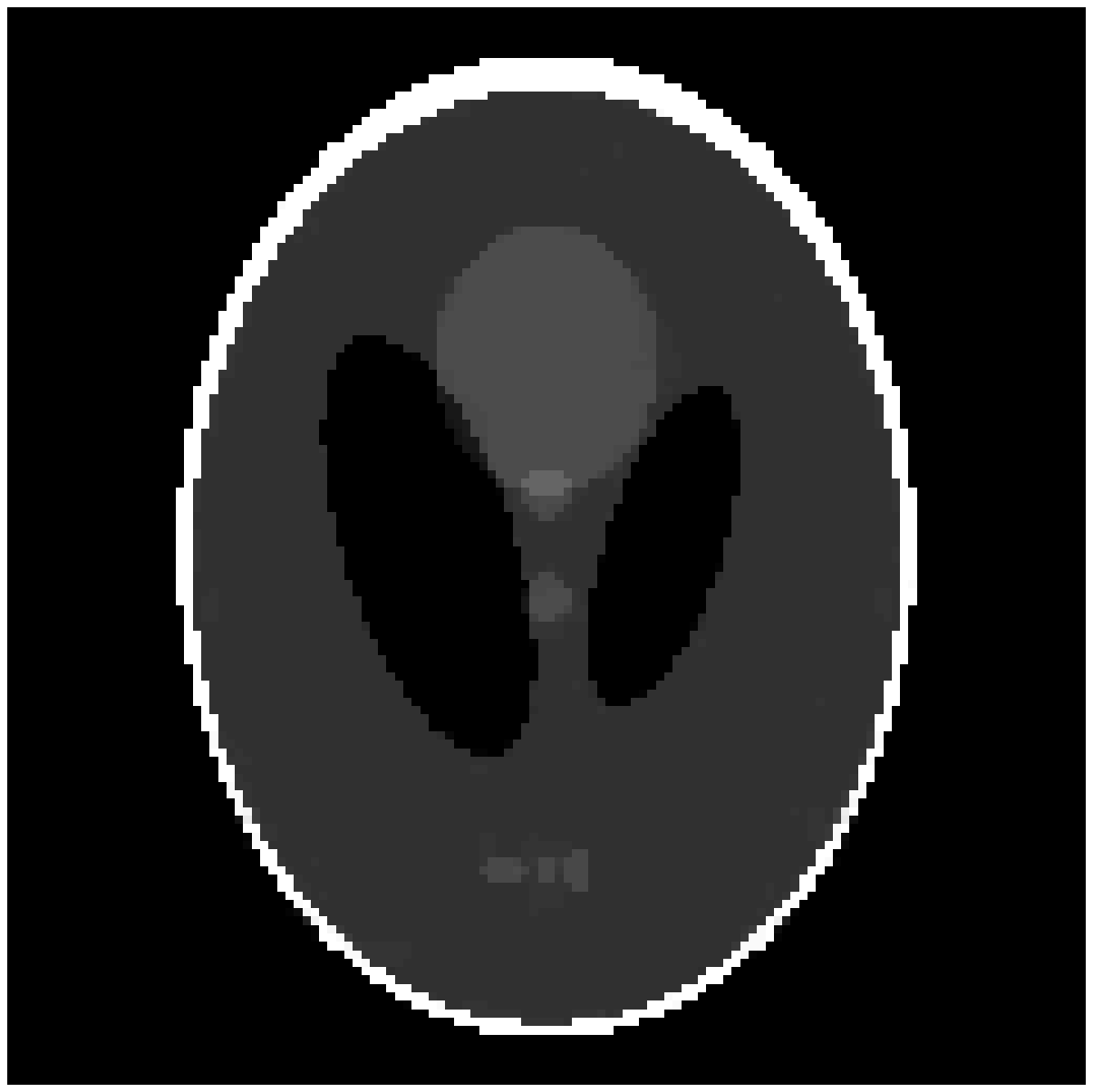} \\
			$\kappa = 10^{-8}$. & $\kappa = 10^{-2}$. \\
			\includegraphics[width=0.4\linewidth]{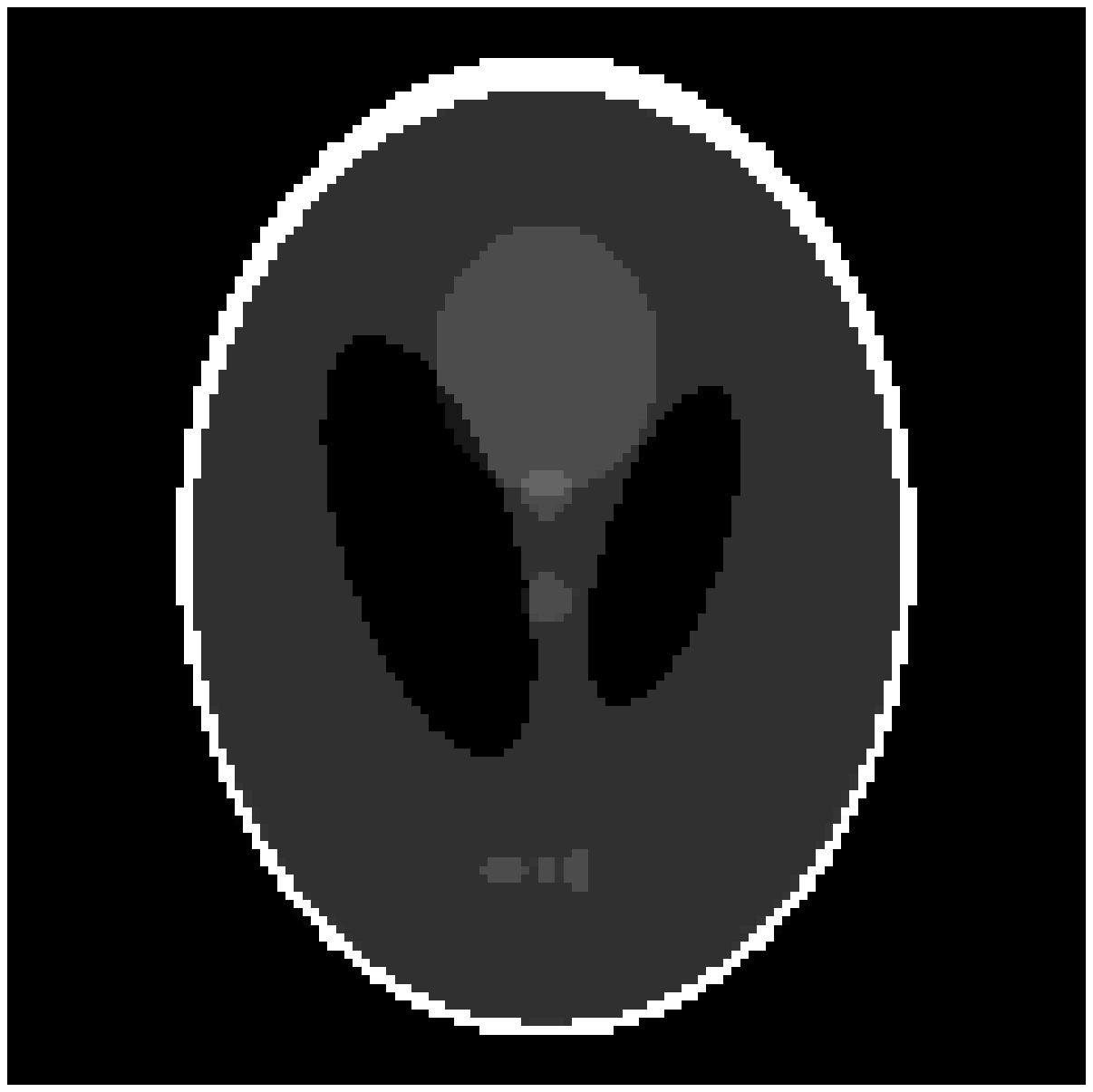} &
			\includegraphics[width=0.4\linewidth]{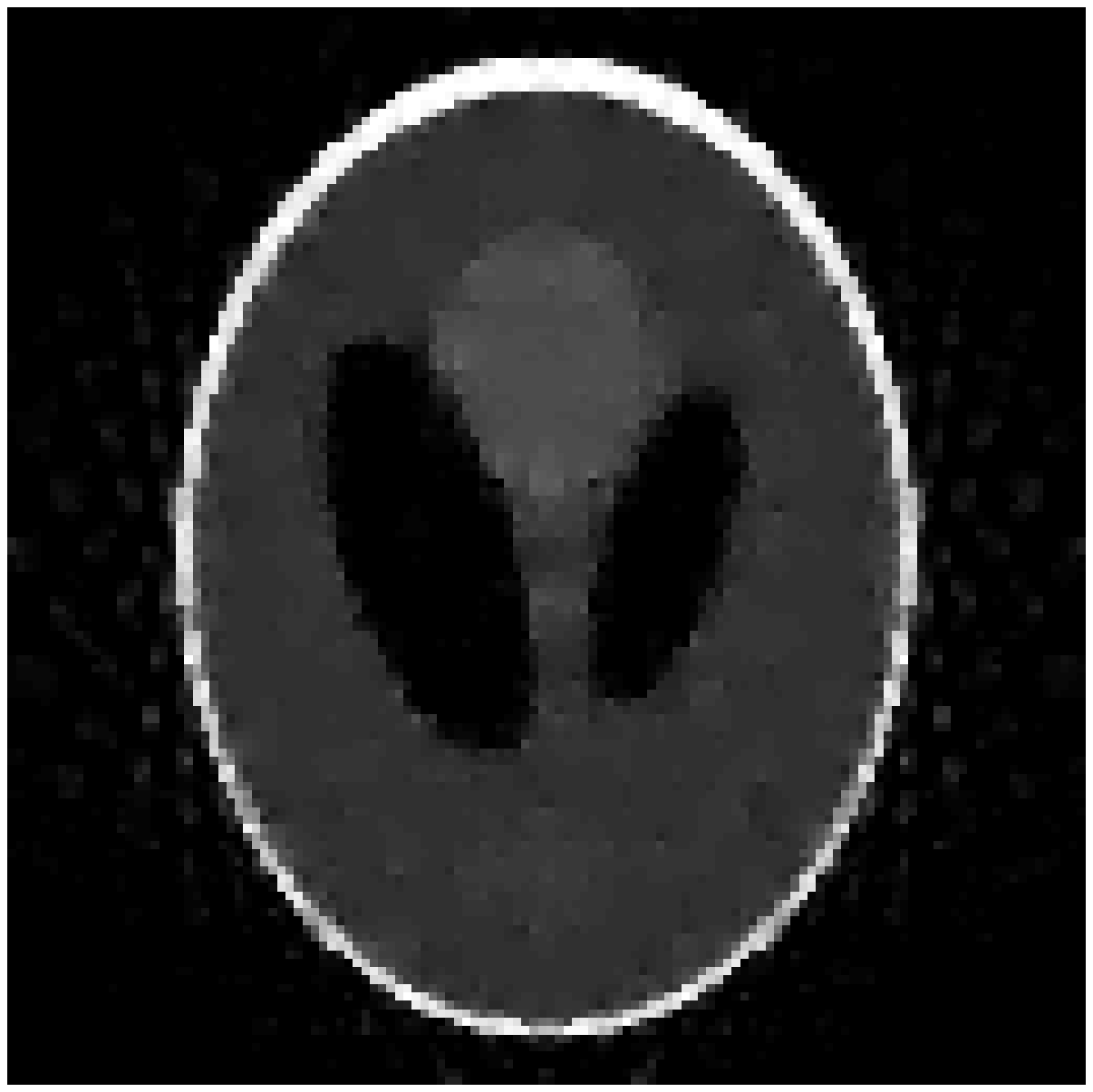} \\
			$\kappa = 10^{0}$. & Standard-form transformation. \\
		\end{tabular}
\caption{Solution to the tomography example for different $\kappa$. We have chosen $\lambda$ by hand to yield the best reconstruction.}
\label{fig:tomo-kappa}
\end{figure}

\section{Conclusion and outlook}
\label{sec:conclusion}
In this paper we have analyzed the method SR3 which was introduced in \cite{Zheng2019}. We have extended theorem 1 from \cite{Zheng2019}   about the singular values of $F_{\kappa}$ to the general form case. We have shown that SR3, as $\kappa\to\infty$, implicitly applies a standard-form transformation, and that for finite $\kappa > 0$, the singular values of $F_{\kappa}$ are related to the standard-form transformed operator. \\
In \cref{sec:value_function} we have shown that the distance between the Pareto curve of the original problem and the Pareto curve of the relaxed problem is of $\mathcal{O}(1/\kappa^2)$ plus the norm of the gradient, which depends on $\kappa$. \\
In \cref{sec:implementation} we have presented our implementation of the inexact SR3 algorithm, where we have proposed an automated stopping criterion for the inner iterations. \\
In our numerical experiments in \cref{sec:experiments} we have compared the SR3 algorithm for two example problems with very different spectra. The gravity example is a severely ill-posed problem and we have shown, numerically, that the convergence of inner iterations is not affected much by $\kappa$, but the convergence of the outer iteration is. For the tomography example we saw a trade-off: as $\kappa$ decreases the outer iterations converge rapidly, but the number of inner iterations is large. We have shown that our automated stopping criterion greatly reduces the number of iterations needed. \\
For future research it would be interesting to further investigate the relation between the Pareto curve of the original problem and of the relaxed problem. Specifically, it would be great if we could prove that the corner of the curves are in the same place, something that we have only been able to show qualitatively through \cref{thm:valuefunction}. This would lead to automatic selection of the regularization parameter $\lambda$. \\
Another interesting topic of research is the selection of $\kappa$. As we have seen in our experiments, the choice of $\kappa$ strongly influences the number of iterations needed for SR3, although this is largely mitigated by the inexact stopping criterion. The relation between the tolerance for the stopping criterion and $\kappa$ should also be further investigated.

\section{Acknowledgements}
The authors would like to thank Dr. Michiel Hochstenbach and Dr. Ajinkya Kadu for fruitful discussions.

\clearpage

\bibliographystyle{siamplain}
\nocite{*}
\bibliography{mybib.bib}
\end{document}